\documentclass[11pt]{amsart}

\usepackage{amsmath,amsfonts}
\usepackage{amssymb}
\usepackage{amscd}
\usepackage{amsthm}
\usepackage{yhmath}
\usepackage{subfigure}
\usepackage{graphicx}   
\usepackage{comment}
\usepackage{stackrel}

\usepackage[normalem]{ulem}
\usepackage{soul}

\usepackage{appendix}

\usepackage[all]{xy}
\usepackage{color}
\usepackage[dvipsnames]{xcolor}

\usepackage{marginnote}
\usepackage{hyperref}

\usepackage{todonotes}

\numberwithin{equation}{section}

\newtheorem{theorem}{Theorem}[section]

\newtheorem{proposition}[theorem]{Proposition}

\newtheorem{definition}[theorem]{Definition}
\newtheorem{lemma}[theorem]{Lemma}
\newtheorem{corollary}[theorem]{Corollary}

\theoremstyle{remark}
\newtheorem{remark}[theorem]{Remark}

\newtheorem{example}[theorem]{Example}

\renewcommand{\phi}{\varphi}

\def\XXint#1#2#3{{\setbox0=\hbox{$#1{#2#3}{\int}$}
	\vcenter{\hbox{$#2#3$}}\kern-.5\wd0}}

  \renewcommand{\a}{\alpha}
  \renewcommand{\b}{\beta}

\newcommand{\ad}{\operatorname{ad}}

\newcommand{\N}{\mathbb N}

\newcommand{\R}{\mathbb R}

\newcommand{\F}{\mathfrak f}
\newcommand{\G}{{\mathfrak g}}
\newcommand{\g}{\mathfrak{g}}

\renewcommand{\H}{\mathcal H}

\newcommand{\I}{{\mathcal I}}

\newcommand{\A}{{\rm Aff}}

\renewcommand{\ker}{\operatorname{Ker}}

\newcommand{\Lie}{\operatorname{Lie}}

\newcommand{\calI}{\mathcal{I}}
\newcommand{\calJ}{\mathcal{J}}
\newcommand{\rank}{\operatorname{rank}}
\newcommand{\Ker}{\operatorname{Ker}}

\newcommand{\Span}{\operatorname{span}}

\renewcommand{\epsilon}{\varepsilon}

\newcommand{\ol}{\overline}

\def\om{\omega}

\newcommand{\aff}{\operatorname{Aff}}
\newcommand{\affh}{\operatorname{Aff_h}}
\newcommand{\affhn}{\operatorname{Aff_h}(\free_n)}
\newcommand{\affhni}{\operatorname{Aff_h}(\free_n)_i}

\newcommand{\anh}{\operatorname{Anh}}
\newcommand{\anhk[1]}{{\operatorname{Anh}^{#1}}}
\newcommand{\calF}{\mathcal{F}}
\newcommand{\frakf}{\mathfrak{f}}
\newcommand{\frakg}{\mathfrak{g}}
\newcommand{\frakh}{\mathfrak{h}}
\newcommand{\fraki}{\mathfrak{i}}

\newcommand{\frakm}{\mathfrak{m}}
\newcommand{\free}{\mathfrak{f}}
\newcommand{\Spt}{\operatorname{N}}
\newcommand{\spt}{\operatorname{im}}

\begin{document} 

\title[Horizontally affine functions on step-2 Carnot algebras]{Horizontally affine functions on step-2 Carnot algebras}

\author{Enrico Le Donne}
\address[Le Donne]{Dipartimento di Matematica, Universit\`a di Pisa, Largo B. Pontecorvo 5, 56127 Pisa, Italy \\
\& \\
University of Jyv\"askyl\"a, Department of Mathematics and Statistics, P.O. Box (MaD), FI-40014, Finland}
\email{enrico.ledonne@unipi.it}

\author{Daniele Morbidelli}
\address[Morbidelli]{Dipartimento di Matematica, Alma Mater Studiorum Universit\`a di Bologna, Italy}
\email{daniele.morbidelli@unibo.it}

\author{S\'everine Rigot}
\address[Rigot]{Universit\'e C\^ote d'Azur, CNRS, LJAD, France}
\email{Severine.RIGOT@univ-cotedazur.fr}

\thanks{E.L.D. was partially supported by the Academy of Finland (grant288501 `\emph{Geometry of subRiemannian groups}' and by grant 322898 `\emph{Sub-Riemannian Geometry via Metric-geometry and Lie- group Theory}') and by the European Research Council (ERC Starting Grant 713998 GeoMeG `\emph{Geometry of Metric Groups}'). S.R. is partially supported by ANR Project SRGI (Sub-Riemannian Geometry and Interactions) ANR-15-CE40-0018.}

\subjclass{20F18, 53C17, 15A75}
%
\keywords{step-2 Carnot groups; step-2 Carnot algebras; horizontally affine functions.}

\begin{abstract}
In this paper we introduce the notion of horizontally affine, h-affine in short, function and give a complete description of such functions on step-2 Carnot algebras. We show that the vector space of h-affine functions on the free step-2 rank-$n$ Carnot algebra is isomorphic to the exterior algebra of $\R^n$. Using that every Carnot algebra can be written as a quotient of a free Carnot algebra, we shall deduce from the free case a description of  h-affine functions on arbitrary step-2 Carnot algebras, together with several characterizations of those step-2 Carnot algebras where  h-affine functions are affine in the usual sense of vector spaces. Our interest for  h-affine functions stems from their relationship with a class of sets called precisely monotone, recently introduced in the literature, as well as from their relationship with minimal hypersurfaces.
\end{abstract}

\maketitle
\setcounter{tocdepth}{2}
\tableofcontents

\section{Introduction} \label{sect:introduction}

 In this paper we introduce the notion of \emph{horizontally affine} function and give a complete description of such functions on step-2 Carnot algebras, or equivalently on step-2 Carnot groups. In the free step-2 rank-$n$ case, we shall see that the vector space of horizontally affine functions is isomorphic to the exterior algebra of $\R^n$. Using the known fact that every step-2 Carnot algebra can be written as a quotient of a free step-2 Carnot algebra, we shall next deduce from the free case a description of horizontally affine functions on arbitrary step-2 Carnot algebras, together with several characterizations of those step-2 Carnot algebras where  h-affine functions are affine in the usual sense of vector spaces.

To introduce the discussion, let us recall some definitions. We refer to Section~\ref{sect:preliminaries} for more details. Let $\frakg=\frakg_1 \oplus \frakg_2$ be a step-2 Carnot algebra, which means that $\frakg$ is a finite dimensional real\footnote{It is worth to stress that our arguments and results can be verbatim extended to finite dimensional nilpotent Lie algebras of step 2 over an arbitrary field of characteristic zero.} nilpotent Lie algebra of step 2, $\frakg_2:=[\frakg,\frakg]$ denotes the derived algebra, and $\frakg_1$ denotes a linear subspace of $\frakg$ that is in direct sum with $\frakg_2$. Such a Lie algebra is naturally endowed with the group law given by
\begin{equation*} 
x\cdot y := x+y+[x,y]
\end{equation*} 
for $x, y \in \frakg$ that makes it a step-2 Carnot group. Actually every step-2 Carnot group can be realized in this way. We shall therefore view a step-2 Carnot algebra both as a Lie algebra and as a Lie  group. We also adopt the notation $y^t:=ty$ for all $t\in\R$ and $y\in\g$. A function $f:\g\to \R$ is  said to be \emph{horizontally affine}, and for brevity we say that $f$ is {\em h-affine} and write $f\in\A_{\rm h} (\G)$, if  for all $x\in \g $ and $y\in \g_1$ the function $t \in \R \mapsto f\left(x\cdot y^t \right)$ is affine. Note that this definition is purely algebraic - it has in particular no connection with the choice of a subRiemannian metric structure on $\frakg$ -- and can be equivalently restated in geometrical terms as follows. A function $f:\g\to \R$ is h-affine if and only if its restriction to each integral curve of every left-invariant horizontal vector field is affine when seen as a function from $\R$ to $\R$, where a left-invariant vector field is said to be horizontal whenever it belongs to $\frakg_1$.

Horizontally affine functions appear naturally in relation with monotone sets, an important  class of sets  introduced by Cheeger and Kleiner~\cite{CheegerKleiner10}, see also \cite{MR2892612,NaorYoung,M1,FasslerOrponenRigot} and the discussion below. However,  h-affine functions are studied systematically for the first time here. See also \cite{antonelli2020polynomial} for a further study of a related notion in more general settings. 
 
Our purposes in the present paper are twofold. We first give a description of h-affine functions on step-2 Carnot algebras, starting with the free case from which the general case will follow. We shall next deduce from this description several characterizations of those step-2 Carnot algebras where  h-affine functions are affine. We shall keep the standard terminology saying that a function $f:\g   \rightarrow \R$ is affine, writing $f\in \A(\g)$, to mean that $f$ is affine in the usual sense considering the vector space structure on $\g$. Note indeed that by elementary properties of step-2 Carnot algebras, each affine function $f\in\aff(\g)$ is   h-affine. In other words, in every step-2 Carnot algebra $\g$, the vector space $\aff(\g)$ is a linear subspace of $\affh(\g)$, see the discussion in Section~\ref{sect:preliminaries}. This inclusion may however be strict, as we shall see.

More explicitly, h-affine functions can be described with the help of the \emph{Carnot dilations} $\delta_t:\g\to\g$ given by $\delta_t(x_1+x_2):=tx_1+t^2 x_2$ for $x_1 \in \frakg_1$, $x_2\in\frakg_2$, and $t\in \R^*:=\R\setminus\{0\}$. Given a non negative integer $i$ we define the vector space of $i$-homogeneous h-affine functions as 
\begin{equation*}
\affh (\frakg)_i := \{f\in \affh (\frakg):\, f\circ \delta_t = t^i f \text{ for all } t \in \R^* \}.
\end{equation*}
We shall prove that every $f\in \affh(\frakg)$ can be written in a unique way as a finite sum of $i$-homogeneous h-affine functions for some $i$'s in $\{0,\dots,\kappa\}$ where $\kappa:=\rank\frakg$ if $\frakg$ is a free step-2 Carnot algebra and $\kappa:=\rank\frakg -1$ if $\frakg$ is a nonfree step-2 Carnot algebra. Recall that the rank of $\frakg$ is defined as $\rank\frakg:=\dim \frakg_1$. Furthermore, denoting by $\Lambda^k\R^n$  the space of alternating $k$-multilinear forms over $\R^n$ (see Section~\ref{sect:appendix-algebra} for our conventions about exterior algebra), we shall also prove that for every $i\in \{0,\dots,\kappa\}$ the vector space $\affh (\frakg)_i$ is isomorphic to a linear subspace of $\Lambda^{\kappa-i} \R^\kappa$. See Theorems \ref{thm:h-affine-maps-free-case}, \ref{thm:h-affine-maps-nonfree-case}, and \ref{thm:h-affine-maps-vs-surjective-Carnot-morphism} for detailed statements.

Let us first consider the free case. Throughout this paper, given an integer $n\geq 2$, we shall use the model for the free step-2 rank-$n$ Carnot algebra $\free_n$ given by 
\begin{equation*}
\free_n := \Lambda^1\R^n \oplus \Lambda^2\R^n
\end{equation*}
equipped with the Lie bracket where the only non trivial relations are given by 
\begin{equation*} 
[\theta,\theta'] := \theta \wedge \theta' \quad \text{ for }\theta, \theta' \in \Lambda^1\R^n.
\end{equation*} 
The induced group law  takes the form
\begin{equation*}
(\theta + \omega) \cdot (\theta' + \omega') := \theta + \theta' + \omega + \omega' + \theta \wedge \theta'
\end{equation*}
for $\theta, \theta' \in \Lambda^1\R^n$, $\omega, \omega' \in \Lambda^2\R^n$. For notational convenience, we shall frequently identify $\free_n$ with $\Lambda^1\R^n\times\Lambda^2\R^n$ writing elements in $\free_n$ as $(\theta,\omega)$ where $\theta \in \Lambda^1\R^n$, $\omega \in \Lambda^2\R^n$.

Given integers $n\geq 2$, $i\in\{0,\dots,n\}$, and $\eta \in \Lambda^{n-i}\R^n$, we define $\varphi_\eta:\free_n \rightarrow \Lambda^n \R^n$ as 
\begin{equation} \label{e:def-phieta} 
\varphi_\eta (\theta,\omega):= \left\{
\begin{aligned}[2]
 &\omega^k \wedge \eta & & \text{ if }  i=2k \text{ is even}
 \\
 &\theta\wedge  \omega^k \wedge \eta & & \text{ if }  i=2k+1 \text{ is odd}.
 \end{aligned}\right.
\end{equation}
The description of h-affine functions on $\free_n$ can then be given in terms of the functions $\varphi_\eta$'s and reads as follows.

\begin{theorem}  \label{thm:h-affine-maps-free-case} 
For $n\geq 2$ we have 
\begin{enumerate}
\item[(i)] $\affhn = \bigoplus_{i=0}^n \affhni$.
\end{enumerate}
Furthermore, given $i\in \{0,\dots,n\}$ and $\nu \in \Lambda^n \R^n \setminus \{0\}$, we have  
\begin{enumerate}
\item[(ii)] $f\in \affhni$ if and only if there is  $\eta\in \Lambda^{n-i}\R^n$, which is unique, such that $f \nu = \varphi_\eta$.
\end{enumerate}
Therefore, for $i\in \{0,\dots,n\}$, the spaces $\affhni$ and $\Lambda^{n-i}\R^n$ are isomorphic as vector spaces, and hence, so are $\affhn$ and $\Lambda^* \R^n$. In particular $\affhn$  is a finite dimensional vector space with dimension $2^n$.
\end{theorem}

Let us briefly explain our strategy to prove Theorem~\ref{thm:h-affine-maps-free-case}. Let $\nu \in \Lambda^n \R^n \setminus \{0\}$ be fixed. It is rather easy to verify that if a function $f:\free_n \rightarrow \R$ is such that $f\nu = \varphi_\eta$ for some $\eta \in \Lambda^{n-i}\R^n$ with $i\in \{0,\dots,n\}$ then $f\in \affhni$, see Lemma~\ref{lem:simple-factors-free-case}. The injectivity of the map $\eta \in \Lambda^{n-i} \R^n \mapsto \varphi_\eta$ is also not hard to verify and follows from general facts about exterior algebra, see Corollary~\ref{cor:injectivity}. The main difficulties are thus to get the decomposition given in Theorem~\ref{thm:h-affine-maps-free-case}~(i), see Theorem~\ref{thm:heisenberg} and Proposition~\ref{prop:step1-free-case}~(i), as well as the fact that every function $f\in \affhn_i$ can be written as $f\nu = \varphi_\nu$ for some $\nu \in \Lambda^{n-i}\R^n$, see Proposition~\ref{prop:step2-free-case}. This will occupy most of Section~\ref{sect:h-affine-maps-free-groups} to which we refer for more details. For the sake of completeness, let us mention the geometric interpretation behind the decomposition in Theorem~\ref{thm:h-affine-maps-free-case} when passing from $\free_n$ for $n\geq 3$ to any Lie subalgebra  of $\free_n$ that is isomorphic to $\frakf_{n-1}$. It can be proved that the zero level set of non-zero $n$-homogeneous h-affine functions on $\free_n$, namely, the set $\{(\theta,\omega)\in \free_n: \omega^{n/2}=0\}$ if $n$ is even, $\{(\theta,\omega)\in \free_n: \theta \wedge \omega^{(n-1)/2}=0\}$ if $n$ is odd, coincides with the union of all Lie subalgebras of $\free_n$ that are isomorphic to $\free_{n-1}$. Therefore, if $f=f_0 + \cdots + f_n \in \affhn$ with $f_i \in \affhni$, one gets that its restriction to any Lie subalgebra isomorphic to $\frakf_{n-1}$ coincides with the restriction to this subalgebra of the sum of the $i$-homogeneous terms $f_0+ \cdots + f_{n-1}$ for $i\in \{0,\dots,n-1\}$ that show up in the decomposition of $f$.

Let us now turn to the general case of arbitrary step-2 Carnot algebras. Our starting point is the known fact that every step-2 Carnot algebra $\g$ can be written as a quotient of  free step-2 Carnot algebras. Namely, by the universal property of free step-2 Carnot algebras, for every $n\geq \rank\frakg$, there is a surjective Carnot morphism $\pi:\free_n\to \g$, see the discussion in Section~\ref{sect:preliminaries}. It turns out that there is a one-to-one correspondence between h-affine functions on $\frakg$ and h-affine functions on $\free_n$ that factor through $\free_n / \Ker \pi$, see Lemma~\ref{lem:h-affine-maps-vs-Carnot-morphisms} and Corollary~\ref{cor:h-affine-maps-i-vs-Carnot-morphisms}. The description of h-affine functions on $\frakg$ can therefore be deduced from the characterization of those functions $\varphi_\eta$ that factor through $\free_n / \Ker \pi$. Namely, we shall verify that for $\eta\in \Lambda^i\R^n$ the function $\varphi_\eta$ factors through $\free_n / \Ker \pi$ if and only if $\eta$ annihilates $\ker\pi$, which means that $\eta \in \anhk[i] \ker \pi$ where
\begin{equation*}
\anhk[i] \ker \pi := \{\eta\in \Lambda^{i} \R^n:\, \eta \wedge \zeta = 0 \text{ for all } \zeta \in \ker\pi\}
\end{equation*}
see Lemma~\ref{lem:factor-through-quotient}. In the genuinely nonfree setting, such a characterization implies the following decomposition of $\affh(\frakg)$.

\begin{theorem} \label{thm:h-affine-maps-nonfree-case}
Let $\frakg$ be a step-2 rank-$r$ Carnot algebra. Assume that $\frakg$ is not isomorphic to $\free_r$. Then  $\affh (\frakg) = \bigoplus_{i=0}^{r-1} \affh (\frakg)_i$.
\end{theorem}

 Note that, in contrast with the free case, one has $\affh (\frakg)_r = \{0\}$ when $\frakg$ is a step-2 rank-$r$ Carnot algebra that is not isomorphic to $\free_r$. This follows from the fact that $\ker\pi \not=\{0\}$, and hence $\anhk[0] \ker \pi = \{0\}$, whenever $\pi : \free_r \rightarrow \frakg$ is a surjective Carnot morphism. A description of the summands $\affh(\frakg)_i$ is provided by the following theorem that applies both in the free and nonfree case (note that when $\frakg = \free_r$ and $n=r$ one recovers Theorem~\ref{thm:h-affine-maps-free-case}) and where the space of annihilators of $\ker\pi$ in $\Lambda^*\R^n$ is defined by
\begin{equation*}
\anh \ker\pi  := \{\eta\in \Lambda^* \R^n:\, \eta \wedge \zeta = 0 \text{ for all } \zeta \in \ker \pi\}~.
\end{equation*}

\begin{theorem} \label{thm:h-affine-maps-vs-surjective-Carnot-morphism}
Let $\frakg$ be a step-2 rank-$r$ Carnot algebra, $n\geq r$, and $\pi:\free_n \rightarrow \frakg$ be a surjective Carnot morphism. Then the following hold true. For $i\in\{0,\cdots,n\}$, $\nu \in \Lambda^n \R^n \setminus \{0\}$,
\begin{enumerate}
\item[(i)] for every $\eta\in \anhk[n-i] \ker \pi$ there is a unique $f\in  \affh (\frakg)_i$ such that $(f \circ \pi) \nu = \varphi_\eta$;
\smallskip
\item[(ii)] for every $f\in  \affh (\frakg)_i$ there is a unique $\eta\in \anhk[n-i] \ker \pi$ such that $(f \circ \pi) \nu = \varphi_\eta$;
\smallskip
\item[(iii)]  via this correspondence, $\affh (\frakg)_i$ and $\anhk[n-i] \ker \pi$ are isomorphic as vector spaces.
\end{enumerate}
 Consequently,
\begin{enumerate}
\item[(iv)] $\affh (\frakg)$ and $\anh \ker \pi$ are isomorphic as vector spaces.
\end{enumerate}
In particular $\affh (\frakg)$ is a finite dimensional vector space.
\end{theorem}

As a consequence of Theorem~\ref{thm:h-affine-maps-vs-surjective-Carnot-morphism} one gets that h-affine functions on step-2 Carnot algebras are polynomials and hence smooth. Let us stress here that there is no regularity assumption in our definition of h-affine functions. Such functions are indeed only assumed to be affine when restricted to horizontal lines, and were not even assumed continuous, nor measurable, beforehand. As a further consequence of their smoothness (one actually only needs local integrability), one can characterize elements in $\affh(\frakg)$ as those locally integrable functions that are harmonic in the distributional sense with respect to every subLaplacian on $\frakg$, see Remark \ref{rmk:subla}. Let us mention that horizontally affine distributions have been recently studied in \cite{antonelli2020polynomial} in wider settings where they can be proved to be polynomials. Note however that this later notion may be different from a pointwise generalization of our present notion of h-affine functions to more general settings, as explained at the end of Remark \ref{rmk:subla}.

Several characterizations of step-2 Carnot algebras~$\g$ where $\affh(\g)=\aff(\g)$ can easily be deduced from Theorem~\ref{thm:h-affine-maps-vs-surjective-Carnot-morphism}, see Theorem~\ref{thm:h-affine-maps-are-affine} below and Section~\ref{sect:h-affine-maps-are-affine}. It turns out that one of these characterizations can be formulated using a class of Lie algebras known in literature as  $\calI$-null, see~\cite{MR3086803}. We recall that a step-2 Carnot algebra $\g=\g_1\oplus\g_2$ is $\calI$-null if every bilinear form $b:\g_1\times\g_2\to\R$ satisfying $b(x, [x,y])=0$ for all $x,y\in\g_1$ vanishes identically on $\g_1\times\g_2$, see Definition~\ref{def:I-null} and Proposition~\ref{prop:I-null}.

\begin{theorem} \label{thm:h-affine-maps-are-affine}
Let $\frakg$ be a step-2 Carnot algebra. Then the following are equivalent:
\begin{enumerate} 
\item[(i)] $\affh (\frakg) = \aff (\frakg)$
\smallskip
\item[(ii)] $\affh (\frakg) = \bigoplus_{i=0}^{2} \affh (\frakg)_i$
\smallskip
\item[(iii)] $\bigoplus_{i\geq 3}\affh (\frakg)_i = \{0\}$, equivalently, $\affh (\frakg)_3 = \{0\}$
\smallskip
\item[(iv)] $\frakg$ is $\calI$-null
\smallskip
\item[(v)] $\bigoplus_{i= 3}^n\anhk[n-i] \ker \pi = \{0\}$, equivalently, $\anhk[n-3] \ker \pi = \{0\}$, for some, equivalently all, $n\geq \max\{3,\rank \frakg\}$, $\pi : \free_n \rightarrow \frakg$ surjective Carnot morphism.
\end{enumerate}
\end{theorem}

Note incidentally that it follows from Theorem~\ref{thm:h-affine-maps-free-case} that $\affh(\free_n) = \aff(\free_n)$ if and only if $n=2$. Therefore the equivalent conditions given in Theorem~\ref{thm:h-affine-maps-are-affine} hold true on $\frakg=\free_n$ if and only if $n=2$. Theorem \ref{thm:h-affine-maps-are-affine} can be efficiently applied in several concrete situations which will be discussed in  Section~\ref{exx} and to which we refer for more details.

Before closing this introduction, we briefly go back to the relationship between h-affine functions and precisely monotone sets, as defined in  \cite{CheegerKleiner10} in the Heisenberg setting. More generally, a subset of a Carnot algebra, identified with a Carnot group, is said to be precisely monotone if the restriction of its characteristic function to each integral curve of every left-invariant horizontal vector field is monotone when seen as a function from $\R$ to $\R$. Equivalently, a precisely monotone set is a h-convex set with h-convex complement, see for instance \cite{Rickly06} for more details about h-convex sets. Precisely monotone sets have been classified in the first Heisenberg algebra $\free_2$, in higher dimensional Heisenberg algebras, and in the direct product $\free_2\times \R$, see 
 \cite{CheegerKleiner10,NaorYoung,M1}. In the aforementioned step-2 settings, it turns out that the boundary of a non empty precisely monotone strict  subset is a hyperplane, while in step-3 Carnot algebras the same statement may be false, see~\cite{Bellettini-LeDonne,MR4228617}. As a consequence of our results, we actually get plenty of examples of Carnot algebras already in the step-2 case where there are precisely monotone subsets whose boundary is not  a hyperplane. Indeed, it can easily be seen that sublevel sets of h-affine functions are precisely monotone. Therefore if $\frakg$ is a step-2 Carnot algebra that is not $\calI$-null and if $f\in \affh(\frakg) \setminus \aff(\frakg)$ then every sublevel set of $f$ is a precisely monotone set whose boundary is not a hyperplane. We refer to the recent paper  \cite{mr} for a more detailed introduction about precisely monotone sets as well as for a classification of such sets in the step-2 rank-3 case in terms sublevel sets of h-affine functions, and for further discussions about higher rank and higher step cases. To conclude these observations, let us mention that measurable precisely monotone sets, and therefore sublevel sets of h-affine functions on step-2 Carnot algebras, can be proved to be local minimizers for the intrinsic perimeter, see \cite[Proposition~3.9]{young2021areaminimizing} and \cite[Proposition~2.9]{mr}.

The  rest of this paper is organized as follows. Section~\ref{sect:preliminaries} contains our conventions and notations about step-2 Carnot algebras and  h-affine functions. We also provide  easy facts that will be useful for later arguments. In Section~\ref{sect:h-affine-maps-free-groups} we focus on the the free case and prove  Theorem~\ref{thm:h-affine-maps-free-case}. Theorems~\ref{thm:h-affine-maps-nonfree-case} and~\ref{thm:h-affine-maps-vs-surjective-Carnot-morphism} are proved in Section~\ref{sect:h-affine-maps-arbitrary-case}. Section~\ref{subsect:h-affine-maps-are-affine} is devoted to the proof of Theorem~\ref{thm:h-affine-maps-are-affine} and Section~\ref{exx} to a discussion of several  examples. In the final Section~\ref{sect:appendix-algebra} we gather notations and facts in linear and exterior algebra.

\textit{Acknowledgements.} The authors are grateful to an anonymous reader of a previous version of this paper for valuable comments and remarks that helped them to improve and simplify the exposition.

\section{Step-2 Carnot algebras and horizontally affine functions} \label{sect:preliminaries} 

We recall that a real\footnote{As already mentioned in the introduction, our arguments and results can be verbatim extended to finite dimensional nilpotent Lie algebras of step 2 over an arbitrary field of characteristic zero.} and finite dimensional Lie algebra $\frakg$ is said to be nilpotent of step 2 if the derived algebra $\frakg_2 := [\frakg,\frakg]$ is non trivial, i.e., $\frakg_2\not=\{0\}$, and central, i.e., $[\frakg,\frakg_2] = \{0\}$. Here, given $U, V \subset \frakg$, we denote by $[U,V]$ the linear subspace of $\frakg$ generated by elements of the form $[u,v]$ with $u\in U$, $v\in V$. If $\frakg_1$ is a linear subspace of $\frakg$ that is in direct sum with $\frakg_2$ then $[\frakg_1,\frakg_1] = \frakg_2$ and the decomposition $\frakg = \frakg_1 \oplus \frakg_2$ is therefore a stratification of $\frakg$. As a matter of fact, every stratification of a nilpotent Lie algebra of step 2 is of this form.

A step-2 Carnot algebra $\frakg$ is a Lie algebra nilpotent of step 2 equipped with a stratification $\frakg = \frakg_1 \oplus \frakg_2$. The rank of $\frakg$ is defined as $\rank \frakg:= \dim \frakg_1$. Such a Lie algebra is naturally endowed with the group law given by 
\begin{equation*}
x\cdot y := x+y+[x,y]
\end{equation*} 
for $x,y \in \frakg$ that makes it a step-2 Carnot group. It is actually well known that any step-2 Carnot group can be realized in this way. We shall therefore view a step-2 Carnot algebra both as a Lie algebra and group.

Throughout this paper, we shall always denote by $\frakg = \frakg_1 \oplus \frakg_2$ a step-2 Carnot algebra. Given $t\in \R$, $x\in \frakg$, we set $x^t := tx$. 

\begin{definition} \label{def:A-affine-maps}
Given $A\subset \frakg$ we say that $f:\frakg\rightarrow \R$ is $A$-affine if for every $x\in \frakg$,  $y \in A$, the function $t\in\R \mapsto f(x\cdot y^t)$ is affine. 
\end{definition}

When $A=\frakg$, one recovers the notion of real valued affine functions on $\frakg$ seen as a vector space. Indeed, since $\frakg$ is nilpotent of step 2, for $x,y \in \frakg$, $t\in \R$, we have $x\cdot y^t = x+t(y+[x,y])$ and $x+ty = x\cdot(y-[x,y])^t$. Therefore $f:\frakg\rightarrow \R$ is $\frakg$-affine if and only if for every $x,y \in \frakg$, $t\in \R$, the function $t\in \R \mapsto f(x+ty)$ is affine, i.e., $f$ is affine, see Proposition~\ref{prop:affine-maps}. In particular, real valued affine functions are $A$-affine for every $A\subset \frakg$. 

In the present paper, we are interested in $\frakg_1$-affine functions, which we shall call \emph{horizontally affine}, \emph{h-affine} in short, namely:

\begin{definition} [h-affine functions] \label{def:h-affine-maps}
We say that $f:\frakg\rightarrow\R$ is \emph{horizontally affine}, \emph{h-affine} in short, if $f$ is $\frakg_1$-affine. In other words, $f$ is h-affine if for every $x \in \frakg$, $y\in \frakg_1$, the function $t\in\R\mapsto f(x\cdot y^t)$ is affine. We denote by $\affh(\frakg)$ the real vector space of  h-affine functions on $\frakg$. 
\end{definition}

We say that $\ell \subset \frakg$ is a horizontal line if there are $x\in \frakg$, $y\in \frakg_1 \setminus \{0\}$ such that $\ell = \{x\cdot y^t:\, t\in\R\}$. We already noticed that horizontal lines are 1-dimensional affine subspaces of $\frakg$ and therefore h-affine functions can equivalently be defined as functions whose restriction to every horizontal line is affine. 

We recall that a Carnot morphism $\pi :\frakg \rightarrow \frakg'$ between step-2 Carnot algebras $\frakg =\frakg_1 \oplus \frakg_2$ and $\frakg' = \frakg'_1 \oplus \frakg'_2$ is a homomorphism of graded Lie algebras, which means that $\pi$ is a linear map such that $\pi([x,y]) = [\pi(x),\pi(y)]$ for all $x,y\in \frakg$ and $\pi(\frakg_k) \subset \frakg'_k$ for $k=1,2$. Note that a Carnot morphism is both a homomorphism of graded Lie algebras and a group homomorphism.

\begin{lemma}  \label{lem:h-affine-maps-vs-Carnot-morphisms}
Let $\frakg, \frakg'$ be step-2 Carnot algebras and $\pi:\frakg \rightarrow \frakg'$ be a Carnot morphism. For $f\in \affh(\frakg')$ we have $f\circ \pi \in \affh(\frakg)$. If $\pi$ is surjective then $f\in \affh(\frakg')$ if and only if $f\circ \pi \in \affh(\frakg)$.
\end{lemma}

\begin{proof} Carnot morphisms map affinely horizontal lines to either horizontal lines or singletons therefore $f\circ \pi \in \affh(\frakg)$ when  $f\in \affh(\frakg')$. If the Carnot morphism  $\pi:\frakg \rightarrow \frakg'$ is surjective then every horizontal line in $\frakg'$ is the affine image through $\pi$ of a horizontal line in $\frakg$ and therefore $f\in \affh(\frakg')$ when $f\circ \pi \in \affh(\frakg)$.
\end{proof}

For $t\in \R^*:= \R \setminus \{0\}$ the Carnot dilation $\delta_t:\frakg \rightarrow \frakg$ is defined as the linear map such that $\delta_t(x) = t^k x$ for $x\in \frakg_k$, $k=1,2$. The family $(\delta_t)_{t\in\R^*}$ is a one parameter group of Carnot automorphisms. Recall  that, given a non negative integer $i$, we denote by $\affh (\frakg)_i:=\{f\in \affh (\frakg):\, f\circ \delta_t = t^i f \text{ for all }  t\in\R^* \}$ the linear subspace of $\affh(\frakg)$ of $i$-homogeneous h-affine functions on $\frakg$. Since dilations commute with Carnot morphisms, we get from Lemma~\ref{lem:h-affine-maps-vs-Carnot-morphisms} the following corollary.

\begin{corollary} \label{cor:h-affine-maps-i-vs-Carnot-morphisms}
Let $\frakg, \frakg'$ be step-2 Carnot algebras, $\pi:\frakg \rightarrow \frakg'$ be a Carnot morphism, and $i$ be a non negative integer. For $f\in \affh(\frakg')_i$ we have $f\circ \pi \in \affh(\frakg)_i$. If $\pi$ is surjective then $f\in \affh(\frakg')_i$ if and only if $f\circ \pi \in \affh(\frakg)_i$.
\end{corollary}

We already noticed that the set $\aff(\frakg)$ of real valued affine functions on $\frakg$ is a linear subspace of $\affh(\frakg)$. More precisely, we have the following inclusion.

\begin{lemma} \label{lem:aff-contained-in-affhi<=2}
$\aff (\frakg)$ is a linear subspace of $\bigoplus_{i=0}^2 \affh (\frakg)_i$.
\end{lemma} 

\begin{proof}
Let $f\in \aff (\frakg)$. There are $f_0\in \R$ and linear forms $f_k:\frakg_k \rightarrow \R$, $k=1,2$, such that $f(x+z) = f_0 + f_1(x) + f_2(z)$ for all $x\in\frakg_1$, $z\in \frakg_2$. Clearly, constant functions belong to $\affh(\frakg)_0$ and the functions $x+z \in \frakg_1\oplus \frakg_2 \mapsto f_1(x)$ and $x+z \in \frakg_1\oplus \frakg_2 \mapsto f_2(z)$ belong  to $\affh(\frakg)_1$ and $\affh(\frakg)_2$ respectively.
\end{proof}

We say that step-2 Carnot algebras are isomorphic if there is a bijective Carnot morphism from one to the other. Note that being h-affine, respectively affine, are intrinsic properties, in particular $\affh(\frakg) = \aff(\frakg)$ if and only if $\affh(\frakg') = \aff(\frakg')$ for isomorphic step-2 Carnot algebras $\frakg, \frakg'$. This indeed more explicitly follows from Lemma~\ref{lem:h-affine-maps-vs-Carnot-morphisms} together with the fact that Carnot morphisms are linear maps.

Let us recall that by the universal property of free step-2 Carnot algebras, see Section~\ref{sect:introduction} for our conventions about the free step-2 rank-$n$ Carnot algebra $\free_n$, given a step-2 rank-$r$ Carnot algebra $\frakg$ and given an integer $n\geq r$, there is a surjective Carnot morphism $\pi:\free_n \rightarrow \frakg$, see for instance~\cite[p.45]{MR1218884}. We also recall that for such a Carnot morphism, $\ker \pi$ is a graded ideal in $\free_n$, which means that $\ker \pi = \fraki_1 \oplus \fraki_2$ where $\fraki_k$ are linear subspaces of $\Lambda^k \R^n$, $k=1,2$, such that $\theta \wedge \theta' \in \fraki_2$ for all $\theta \in \Lambda^1\R^n$, $\theta'\in \fraki_1$.

We now recall the definition of $\calI$-null Lie algebras that will be used in one of our  characterizations of those step-2 Carnot algebras where h-affine functions are affine, see Theorem~\ref{thm:h-affine-maps-are-affine}.

\begin{definition}[\cite{MR3086803}] \label{def:I-null}
A Lie algebra $\frakm$ is said to be $\calI$-null if for every symmetric bilinear invariant form $B:\frakm\times \frakm \rightarrow \R$ we have $B(\frakm, [\frakm,\frakm])= 0$.  Here $B$ is said to be invariant if $B(x,[y,z]) = B([x,y],z)$ for all $x,y,z \in \frakm$, or equivalently, if the trilinear form $B(\cdot,[\cdot,\cdot])$ is alternating on $\frakm \times \frakm \times \frakm$. 
\end{definition}

For step-2 Carnot algebras, the previous definition can be rephrased in the following way, of which we omit the elementary proof.

\begin{proposition} \label{prop:I-null} 
A step-2 Carnot algebra $\frakg = \frakg_1 \oplus \frakg_2$ is $\calI$-null if and only if every bilinear form $b:\g_1\times\g_2\to\R$ satisfying $b(x, [x,y])=0$ for all $x,y\in\g_1$ vanishes identically on $\g_1\times\g_2$.
\end{proposition}

\begin{remark}\label{rmk:subla} 
In the present article, we focus on step-2 Carnot algebras or, equivalently, step-2 Carnot groups. Let us mention that the notion of horizontally affine function makes sense in broader generality. One may for instance consider Carnot groups of arbitrary step (see~\cite{MR3587666},~\cite{LeDonne:Carnot} for a primer on the subject) or, more generally, a connected nilpotent Lie group $G$ equipped with a  vector  subspace $\Delta$ of its Lie algebra $\frakg$ that Lie generates $\frakg$. Then we say that $f: G \to \R$ is  $\Delta$-affine if for every $X\in \Delta$ the restriction of $f$ to each integral curve of $X$ is affine when seen as a function from $\R$ to $\R$. Here an element $X\in \Delta$ is seen as a left-invariant vector field on $G$. When $G$ is a step-2 Carnot group with stratified Lie algebra $\frakg = \frakg_1 \oplus \frakg_2$ and $\Delta=\frakg_1$ is the first, usually called horizontal, layer of the stratification of $\frakg$, one recovers Definition~\ref{def:h-affine-maps}, and this latter definition can hence be extended to Carnot groups of arbitrary step in the obvious way. Going back to the aforementioned more general setting and considering $G$ equipped with a Haar measure, let us mention that we have the following characterizations of locally integrable $\Delta$-affine functions. Namely, $f \in L^1_{\text{loc}}(G)$  is $\Delta$-affine if and only if one of the following equivalent conditions holds true in the distributional sense: 
\begin{enumerate}
\item[(A.1)] $X^2 f=0$ for every $X\in \Delta$
\smallskip
\item[(A.2)] $X Y f + Y X f = 0$ for every $X, Y \in \Delta$
\smallskip
\item[(A.3)] $X_1^2 f + \cdots+ X_m^2 f =0$ for every basis $(X_1, \ldots, X_m)$ of $\Delta$.
\end{enumerate}
Indeed, if $f \in L^1_{\text{loc}}(G)$  is $\Delta$-affine then (A.1) holds true as a consequence of the very definitions. Conversely, if $f \in L^1_{\text{loc}}(G)$ satisfies (A.1) then $f$ is smooth by H\"ormander's hypoellipticity theorem and it then clearly follows from (A.1) that $f$ is $\Delta$-affine. The fact that (A.1) is equivalent to (A.2) is a consequence of H\"ormander's hypoellipticity theorem together with the identity $(X+Y)^2 f =   X^2 f +XY f +Y Xf +Y^2 f$ for smooth functions $f$. Condition~(A.1) obviously implies (A.3). Conversely, if $f \in L^1_{\text{loc}}(G)$ satisfies (A.3) and $X\in\Delta\setminus\{0\}$, one can complete $X$ into a basis  $X,X_2, \dots, X_m$ of $\Delta$. Then for every $\epsilon>0$ one has $X^2 f + \epsilon X_2^2 f + \cdots+ \epsilon X_m^2 f=0$ with the left-handside converging to $X^2 f$ as $\epsilon \to 0$ and therefore $X^2 f=0$. See also~\cite{antonelli2020polynomial} for other generalizations of condition~(A.1) for locally integrable functions.

To conclude this remark, note that in the specific setting considered in this paper, i.e., step-2 Carnot algebras $\frakg = \frakg_1 \oplus \frakg_2$, it follows from Theorem~\ref{thm:h-affine-maps-vs-surjective-Carnot-morphism} that h-affine functions are smooth and hence locally integrable. Therefore each of the distributional sense conditions (A.1), (A.2), (A.3) with $\Delta=\frakg_1$ makes sense for all h-affine functions and hence characterizes such a class of functions. In the more general setting considered in the present remark, it is however not clear to us whether $\Delta$-affinity implies local integrability, and the class of locally integrable $\Delta$-affine functions, that can be characterized through each of the equivalent conditions (A.1), (A.2), (A.3), could therefore be a strict subset of the class of $\Delta$-affine functions.
\end{remark}


\section{Horizontally affine functions on free step-2 Carnot algebras} \label{sect:h-affine-maps-free-groups}

This section is devoted to the proof of Theorem~\ref{thm:h-affine-maps-free-case}. The proof will proceed into 4 steps. We first verify in Lemma~\ref{lem:simple-factors-free-case} that for $\eta \in \Lambda^{n-i} \R^n$ we have $\varphi_\eta \in \affh(\free_n ,\Lambda^n\R^n)_i$ where $\varphi_\eta$ is given by~\eqref{e:def-phieta}, together with the injectivity of the linear map $\eta \in \Lambda^{n-i} \R^n \mapsto \varphi_\eta \in \affh(\free_n ,\Lambda^n\R^n)_i$.  We shall next prove Theorem~\ref{thm:h-affine-maps-free-case} for $n=2$, see Theorem~\ref{thm:heisenberg}, and deduce properties of h-affine functions on $\free_n$ for $n\geq 3$ to be used in the next step, see Proposition~\ref{prop:Sigma-affine}. When $n\geq 3$, we first  prove  that $\affhn = \bigoplus_{i=0}^n \affhni$ together with preliminary information about elements in $\affhni$, see Proposition~\ref{prop:step1-free-case}. We then upgrade these information in Proposition~\ref{prop:step2-free-case} to get the description stated in Theorem~\ref{thm:h-affine-maps-free-case}.

For notational convenience, we identify in this section $\free_n$ with $\Lambda^1\R^n \times \Lambda^2\R^n$ and write elements in $\free_n$ as $x=(\theta,\omega)$ with $\theta \in \Lambda^1\R^n$, $\omega\in\Lambda^2\R^n$. In the next lemma, we denote by $\affh(\free_n ,\Lambda^n\R^n)_i$ the analogue of $\affhni$ for $\Lambda^n\R^n$-valued functions. More explicitly, $f:\free_n \rightarrow\Lambda^n\R^n$ belongs to $\affh(\free_n ,\Lambda^n\R^n)_i$ if and only if for every $(\theta,\omega)\in \free_n$,  $\theta' \in \Lambda^1\R^n$, the function $t\in\R \mapsto f((\theta,\omega)\cdot (t\theta',0)) \in \Lambda^n\R^n$ is affine, and $f\circ \delta_t = t^i f$ for all $t\in \R^*$.

\begin{lemma} \label{lem:simple-factors-free-case}
For $n\geq 2$, $i\in \{0,\dots,n\}$, and $\eta\in \Lambda^{n-i} \R^n$, we have $\varphi_\eta \in \affh(\free_n ,\Lambda^n\R^n)_i$ where $\varphi_\eta$ is given by~\eqref{e:def-phieta}. Furthermore, the linear map $\eta \in \Lambda^{n-i} \R^n \mapsto \varphi_\eta \in \affh(\free_n ,\Lambda^n\R^n)_i$ is injective.
\end{lemma}

\begin{proof}
Let $\eta\in \Lambda^{n-i} \R^n$. Clearly $\varphi_\eta \circ \delta_t = t^i \varphi_\eta$ for all $t\in \R^*$. If $i=2k$ is even, we have
\begin{equation*}
\varphi_\eta((\theta,\omega)\cdot (t\theta',0)) = (\omega+t\theta\wedge \theta')^k \wedge \eta = \omega^k\wedge \eta + t k\omega^{k-1}\wedge\theta\wedge\theta'\wedge\eta,
\end{equation*}
if $i=2k+1$ is odd,
\begin{equation*}
\varphi_\eta((\theta,\omega)\cdot (t\theta',0)) = (\theta+t\theta') \wedge (\omega+t\theta\wedge \theta')^k \wedge \eta = \theta \wedge \omega^k \wedge \eta + t \theta' \wedge \omega^k \wedge \eta,
\end{equation*}
for all $(\theta,\omega)\in \free_n$, $\theta'\in \Lambda^1\R^n$.  Therefore $\varphi_\eta \in \affh(\free_n,\Lambda^n\R^n)_i$. For the injectivity of the linear map $\eta \in \Lambda^{n-i} \R^n \mapsto \varphi_\eta \in \affh(\free_n ,\Lambda^n\R^n)_i$, see Corollary~\ref{cor:injectivity}.
\end{proof}

\begin{theorem} \label{thm:heisenberg}
We have $\affh(\free_2) = \aff(\free_2)$.
\end{theorem}

\begin{proof}
We recall that a set $\ell \subset \free_2$ is said to be a horizontal line if $\ell=\{(\theta,\omega)\cdot (t\theta',0) : t\in\R\}$ for some $(\theta,\omega)\in \free_2$, $\theta'\in \Lambda^1\R^2 \setminus\{0\}$. Define the h-affine hull of a set $A\subset \free_2$ as the smallest set $C$ containing $A$ with the property that if a horizontal line $\ell$ meets $C$ in more than one point then $\ell \subset C$. It follows from~\cite[Lemma~4.10]{CheegerKleiner10} that there are 4 points in $\free_2$ whose h-affine hull is $\free_2$. Indeed, given linearly independent $\theta,\theta' \in \Lambda^1\R^2$, the h-affine hull $C$ of $\{(0,0), (\theta,0), (\theta',0),(\theta+\theta',\theta \wedge \theta')\}$ contains a pair of parallel lines with distinct projection in the sense of~\cite{CheegerKleiner10}, namely, the horizontal line through $(0,0)$ and $(\theta',0)$ and the horizontal line through $(\theta,0)$ and $(\theta+\theta',\theta \wedge \theta')$, therefore $C =\free_2$ by~\cite[Lemma~4.10]{CheegerKleiner10}. This implies that $\affh(\free_2)$ is a vector space with dimension $\leq$ 4. Since $\aff(\free_2)$ is a 4-dimensional linear subspace of $\affh(\free_2)$, we get that $\affh(\free_2) = \aff(\free_2)$, as claimed.
\end{proof}

Note that Theorem~\ref{thm:h-affine-maps-free-case} for $n=2$ follows from Lemma~\ref{lem:aff-contained-in-affhi<=2}, Lemma~\ref{lem:simple-factors-free-case} and Theorem~\ref{thm:heisenberg}. For $n\geq 3$ we set $\Sigma_n: = \bigcup_{\theta,\theta'\in \Lambda^1\R^n} \Lie(\theta,\theta')$ where  $\Lie (\theta,\theta'):=\Span \{\theta,\theta'\} \times \Span \{\theta \wedge \theta'\}$ denotes the Lie subalgebra of $\free_n$ generated by $\theta, \theta'\in \Lambda^1\R^n$. We refer to Definition~\ref{def:A-affine-maps} for the definition of $\Sigma_n$-affine functions.

\begin{proposition} \label{prop:Sigma-affine} For $n\geq 3$, $f\in \affhn$, the following hold true:
\begin{gather} 
f \text{ is }  \Sigma_n\text{-affine}, \label{e:Sigma-affine}\\
f_\omega \in \aff(\Lambda^1\R^n)  \text{ for all } \omega \in \Lambda^2\R^n, \label{e:f_omega}
\end{gather}
where $f_\omega:   \Lambda^1\R^n \rightarrow \R$ is given by  $f_\omega(\theta):= f(\theta,\omega)$ and $\aff(\Lambda^1\R^n)$ denotes the space of real valued affine functions on $\Lambda^1\R^n$.
\end{proposition}

\begin{proof}
Clearly, composing h-affine functions with left-translations yields h-affine functions. Therefore, to prove that every $f\in \affhn$ is $\Sigma_n$-affine, we only need to verify that for every $f\in \affhn$, $\theta_1,\theta_2 \in \Lambda^1\R^n$, $\theta_1 \wedge \theta_2 \not=0$, $(\theta,\omega)\in \Lie(\theta_1,\theta_2)$, the function $t\in\R \mapsto f(t\theta,t\omega)$ is affine. Set $\frakh:=\Lie(\theta_1,\theta_2)$ and denote by $f_\frakh$ the restriction of $f$ to $\frakh$. On the one hand, the structure of step-2 Carnot algebra of $\free_n$ induces on $\frakh$ a structure of step-2 Carnot algebra that makes it isomorphic to $\free_2$. Therefore $\affh(\frakh) = \aff(\frakh)$ by Theorem~\ref{thm:heisenberg}. On the other hand, $f_\frakh \in \affh(\frakh)$. Thus $f_\frakh \in \aff(\frakh)$, which implies that for all $(\theta,\omega)\in \frakh$, the function $t\in\R \mapsto f(t\theta ,t\omega)$ is affine, and concludes the proof of~\eqref{e:Sigma-affine}. To prove~\eqref{e:f_omega}, note that for $\omega\in \Lambda^2\R^n$, $\theta,\theta'\in \Lambda^1\R^n$, $t\in\R$, we have $(\theta + t\theta',\omega) = (\theta,\omega)\cdot (t\theta',t\theta'\wedge \theta)$. Since $(\theta',\theta'\wedge \theta) \in \Sigma_n$, it follows from Proposition~\ref{prop:affine-maps} that $f_\omega \in \aff(\Lambda^1\R^n)$ for every $\Sigma_n$-affine function $f$, and hence, in particular for $f\in \affhn$ by~\eqref{e:Sigma-affine}.
\end{proof}

In addition to the notations given in the appendix, see Section~\ref{sect:appendix-algebra}, we shall use the following ones in the rest of this section. Recall that $(e^1,\dots,e^n)$ denotes a basis of $\Lambda^1\R^n$. For $\theta = \sum_{j=1}^n \theta_j\, e^j \in \Lambda^1\R^n$ we set $\theta_J:= \theta_{j_1}\cdots\theta_{j_k}$ for $J=(j_1,\dots,j_k) \in \calJ_{k}^{n}$ with the convention $\theta_\emptyset:= 1$.

We set 
\begin{equation*}
\calI:=\left\{\alpha \beta: \alpha,\beta\in \N\setminus\{0\},\, \alpha < \beta \right\} , 
\end{equation*}
and we equip $\calI$ with the lexicographic order, i.e., we write $\alpha\beta <\alpha'\beta'$ to mean either that $\alpha = \alpha'$ and $\beta<\beta'$ or that $\alpha < \alpha'$. We set $\overline{\calI}_{0}^{\, n}:=\{\emptyset\}$, 
\begin{equation*}
\overline{\calI}_{k}^{\, n}:=\left\{(\alpha_1 \beta_1, \dots,\alpha_k \beta_k)\in \calI^k:\, 12 \leq \alpha_1 \beta_1 < \cdots < \alpha_k \beta_k \leq (n-1) n\right\}
\end{equation*}
for $k\in\{1,\dots,n(n-1)/2\}$, and $\overline{\calI}^{\, n}:= \cup_{0\leq k \leq n(n-1)/2}\, \overline{\calI}_{k}^{\, n}$. We write $\spt \emptyset := \emptyset$ and $\spt I:=\{\alpha_1\beta_1,\dots,\alpha_k\beta_k\} \subset\calI$ for $I=(\alpha_1\beta_1,\dots,\alpha_k\beta_k) \in \overline{\calI}_{k}^{\, n}$. Given $I,I'\in \overline{\calI}^{\, n}$, we denote by $I\setminus I'\in\overline{\calI}^{\, n}$ the unique element in $\overline{\calI}^{\, n}$ such that $\spt (I \setminus I') = \spt I \setminus \spt I'$ and we write $I'\subset I$ to mean that $\spt I' \subset \spt I$.

For $\omega = \sum_{\alpha\beta \in \overline{\calI}_{1}^{\, n}\cong \calJ_{2}^{n}} \omega_{\alpha\beta} \, e^{\alpha\beta} \in \Lambda^2\R^n$ we set $\omega_\emptyset:= 1$ and $\omega_I:= \omega_{\alpha_1\beta_1}\cdots\omega_{\alpha_k\beta_k}$ for $I=(\alpha_1\beta_1,\dots,\alpha_k\beta_k) \in \overline{\calI}_{k}^{\, n}$.

We write $\Spt (\emptyset) := \emptyset$ and $\Spt(I):=\{\alpha_1,\beta_1,\dots,\alpha_k,\beta_k\} \subset \N $ for $I=(\alpha_1\beta_1,\dots,\alpha_k\beta_k) \in \overline{\calI}_{k}^{\, n}$. We set $\calI_{0}^{n}:=\{\emptyset\}$ and
\begin{equation*}
\calI_{k}^{n}:=\left\{(\alpha_1 \beta_1, \dots,\alpha_k \beta_k) \in \overline{\calI}_{k}^{\, n}:\, \Spt (\alpha_i\beta_i) \cap \Spt (\alpha_j\beta_j) = \emptyset \text{ for all } i\not=j \right\}~.
\end{equation*}
for $k\in\{1,\dots,n(n-1)/2\}$. Note that $\overline{\calI}_{0}^{\, n}=\calI_{0}^{n}$, $\overline{\calI}_{1}^{\, n}=\calI_{1}^{n}$. When $n\geq 3$, we have $\calI_{k}^{n} \varsubsetneq\overline{\calI}_{k}^{n}$ for $2\leq k \leq n(n-1)/2$, and $\calI_{k}^{n} =\emptyset$ for $k>\lfloor n/2\rfloor$.

\medskip

\begin{proposition} \label{prop:step1-free-case}
For $n\geq 3$ the following holds true:
\begin{enumerate}
\item[(i)] $\affhn = \bigoplus_{i=0}^n \affhni$,
\smallskip
\item[(ii)] for $k\in\{0,\dots,\lfloor n/2 \rfloor\}$, every $f\in \affhn_{2k}$  can be written as $$(\theta,\omega) \mapsto \sum_{I\in \calI_{k}^{n}} a_I \, \omega_I $$ for constants $a_I \in \R$,
\smallskip
\item[(iii)] for $k\in\{0,\dots,\lfloor (n-1)/2 \rfloor\}$, every $f\in \affhn_{2k+1}$  can be written as
$$(\theta,\omega) \mapsto  \sum_{I\in \calI_{k}^{n}} b_I(\theta) \, \omega_I~$$
for linear forms $b_I: \Lambda^1\R^n \rightarrow \R$.
\end{enumerate}
\end{proposition}

\begin{proof}
For $i\in \{1,\dots,n\}$ the $\affhni$ are linear subspaces of $\affhn$ that are in direct sum. Therefore $\bigoplus_{i=0}^n \affhni \subset \affhn$. Conversely, let $f\in \affhn$ be given.

We first prove that there are functions $c_I: \Lambda^1\R^n \rightarrow \R$, $I \in \overline{\calI}^{\, n}$, such that
\begin{equation} \label{e:step1}
f(\theta,\omega) = \sum_{I \in \overline{\calI}^{\, n}} c_I(\theta) \, \omega_I~.
\end{equation}
Let $\theta\in \Lambda^1\R^n$ be given. We know from~\eqref{e:Sigma-affine} that, for every $\omega\in \Lambda^2\R^n$, $1\leq \alpha<\beta\leq n$, the function $t\in\R \mapsto f((\theta,\omega)\cdot (0, t e^{\alpha\beta}))$ is affine. Since $f((\theta,\omega)\cdot (0, t e^{\alpha\beta})) = f(\theta,\omega+ t e^{\alpha\beta})$, it follows from elementary properties of multiaffine maps, see Proposition~\ref{prop:homogeneous-polynomial} applied to $\omega \in \Lambda^2\R^n \mapsto f(\theta,\omega)$, that there are $c_I(\theta) \in \R$, $I \in \overline{\calI}^{\, n}$, such that~\eqref{e:step1} holds true.

Next, we prove that 
\begin{equation} \label{e:step2}
c_I \in \aff(\Lambda^1\R^n) \quad \text{for all } I\in\overline{\calI}^{\, n}~.
\end{equation}
We have $\overline{\calI}^{\, n} = \cup_{0\leq k \leq n(n-1)/2}\, \overline{\calI}_{k}^{\, n}$ and we prove by induction on $k$ that $c_I \in \aff(\Lambda^1\R^n)$ for all $I\in\overline{\calI}_{k}^{\, n}$. For $k=0$ we have $\overline{\calI}_{0}^{\, n}=\{\emptyset\}$ with $c_\emptyset(\theta) = f(\theta,0)$ and we apply \eqref{e:f_omega} with $\omega = 0$  to get that $c_\emptyset \in \aff(\Lambda^1\R^n)$. Given $k\in\{1,\dots,n(n-1)/2\}$, assume that $c_{I'} \in \aff(\Lambda^1\R^n)$ for all $I'\in \cup_{0\leq i \leq k-1} \overline{\calI}_{i}^{\, n}$. For $I=(\alpha_1\beta_1,\dots,\alpha_k\beta_k)\in \overline{\calI}_{k}^{\, n}$, we apply \eqref{e:f_omega} with $\omega = \sum_{j=1}^k e^{\alpha_j\beta_j}$ to get that $$c_I \,+ \sum_{\substack{I'\in \cup_{0\leq i \leq k-1} \overline{\calI}_{i}^{\, n} \\ I'\subset I}} c_{I'}   \, \in \aff(\Lambda^1\R^n)~.$$ By induction hypothesis, we get that $c_I\in \A(\Lambda^1\R^n)$, which concludes the proof of~\eqref{e:step2}.

It follows from~\eqref{e:step2} that there are constants $a_I\in\R$ and linear forms  $b_I: \Lambda^1\R^n\to \R$, $I \in \overline{\calI}^{\, n}$, such that $c_I(\theta) = a_I + b_I(\theta)$ for every $\theta\in\Lambda^1\R^n$. For $k\in\{0,\dots,n(n-1)/2\}$, we set
\begin{equation*} 
f_{2k}(\theta,\omega) := \sum_{I\in \overline{\calI}_{k}^{\, n}} a_I \, \omega_I \quad \text{and} \quad f_{2k+1}(\theta,\omega) := \sum_{I\in \overline{\calI}_{k}^{\, n}} b_I(\theta) \, \omega_I~,
\end{equation*}
so that $f= \sum_{i=0}^{n(n-1)+1} f_i$. We claim that $f_i \in \affhni$ for all $i\in\{0,\dots,n(n-1)+1\}$. Indeed, let $(\theta,\omega) \in \free_n$, $\theta'\in \Lambda^1\R^n$, $s\in\R$ be given. Since the dilations are Carnot automorphisms, we know from Lemma~\ref{lem:h-affine-maps-vs-Carnot-morphisms} that $f\circ \delta_t \in \affhn$ for all $t\in\R^*$. Therefore
\begin{equation*}
\begin{split}
\sum_{i=0}^{n(n-1)+1} t^i f_i ((\theta,\omega)\cdot (s\theta',0)) &= (f\circ \delta_t) ((\theta,\omega)\cdot (s\theta',0)) \\
&= (f\circ \delta_t) (\theta,\omega) + s ((f\circ \delta_t) ((\theta,\omega)\cdot (\theta',0)) - (f\circ \delta_t) (\theta,\omega))\\
&= \sum_{i=0}^{n(n-1)+1}  t^i (f_i(\theta,\omega) + s (f_i ((\theta,\omega)\cdot (\theta',0)) - f_i (\theta,\omega)))
\end{split}
\end{equation*}
for all $t\in\R^*$, which implies that $f_i ((\theta,\omega)\cdot (s\theta',0)) = f_i(\theta,\omega) + s (f_i ((\theta,\omega)\cdot (\theta',0)) - f_i (\theta,\omega))$ for all $i\in\{0,\dots,n(n-1)+1\}$. Since this holds true for all $(\theta,\omega) \in \free_n$, $\theta'\in \Lambda^1\R^n$, $s\in\R$, we get that $f_i \in \affhn$. Clearly, we also have $f_i\circ \delta_t = t^i f_i$ for all $t\in \R^*$. Therefore $f_i \in \affhni$ for all $i\in\{0,\dots,n(n-1)+1\}$, as claimed.

For $k\in \{0,\dots,n(n-1)/2\}$, we now claim that 
\begin{equation} \label{e:vanishing-a_I}
a_I = 0 \quad \text{for all } I \in \overline{\calI}_{k}^{\, n} \setminus \calI_{k}^{n}~.
\end{equation}
For $k\in\{0,1\}$, we have $\overline{\calI}_{k}^{\, n} = \calI_{k}^{n}$ and there is nothing to prove.   Let $k\in \{2,\dots,n(n-1)/2\}$ be given. First, note that $I \in \overline{\calI}_{k}^{\, n} \setminus \calI_{k}^{n}$ if and only if there are integers $1\leq \alpha < \beta < \gamma \leq n$ such that either $(\alpha\beta,\alpha\gamma) \subset I$, or $(\alpha\gamma,\beta\gamma)\subset I$, or $(\alpha\beta,\beta\gamma)\subset I$. Now, let $1\leq \alpha < \beta < \gamma \leq n$ be given. On the one hand, since $f_{2k} \in \affhn$, we know that, for every $\omega\in \Lambda^2\R^n$, the function $t\in\R \mapsto f_{2k}((e^\alpha,\omega)\cdot (t(e^\beta + e^\gamma),0) =  f_{2k}(0,\omega + t ( e^{\alpha\beta} + e^{\alpha\gamma}))$ is affine. On the other hand, this function is a polynomial for which the coefficient of $t^2$, namely, 
\begin{equation*}
\sum_{\substack{I\in \overline{\calI}_{k}^{\, n} \\ (\alpha\beta,\alpha\gamma) \subset I}} a_I \, \omega_{I\setminus (\alpha\beta,\alpha\gamma)} 
\end{equation*}
must therefore vanish. Since this holds true for every $\omega \in \Lambda^2\R^n$, it follows that $a_I = 0$ for every $I\in \overline{\calI}_{k}^{\, n}$ such that $(\alpha\beta,\alpha\gamma) \subset I$. Considering the function $t\in \R \mapsto f_{2k}(0,\omega + t  (e^{\alpha\gamma}  + e^{\beta\gamma}))$, respectively $t\in \R \mapsto f_{2k}(0,\omega + t  (e^{\alpha\beta} + e^{\beta\gamma}))$, and arguing in a similar way, we get that $a_I = 0$ for every $I\in \overline{\calI}_{k}^{\, n}$ such that $(\alpha\gamma,\beta\gamma)\subset I$, respectively such that $(\alpha\beta,\beta\gamma)\subset I$, which concludes the proof of~\eqref{e:vanishing-a_I}.

For $k\in \{0,\dots,n(n-1)/2\}$, we claim that 
\begin{equation} \label{e:vanishing-b_I}
 b_I (\theta) = 0  \quad \text{for all } \theta \in \Lambda^1\R^n,\, I \in \overline{\calI}_{k}^{\, n} \setminus \calI_{k}^{n}~.
\end{equation}
Indeed, let $k\in \{0,\dots,n(n-1)/2\}$, $\theta \in \Lambda^1\R^n$ be given. Consider the function $g:\free_n \rightarrow \R$ given by $g(\tau,\omega):=f_{2k+1}(\theta,\omega)$ for $(\tau,\omega)\in \free_n$. We have $g((\tau,\omega)\cdot (t\tau',0)) = f_{2k+1}((\theta,\omega)\cdot(0,t\tau\wedge\tau'))$ for all $(\tau,\omega) \in \free_n$, $\tau'\in \Lambda^1\R^n$. Since $f_{2k+1} \in \affhn$, it follows from~\eqref{e:Sigma-affine} that $g\in \affhn$. We then argue as for the proof of~\eqref{e:vanishing-a_I} to get that $b_I(\theta) = 0$ for every $I \in \overline{\calI}_{k}^{\, n} \setminus \calI_{k}^{n}$, which concludes the proof of~\eqref{e:vanishing-b_I}.

For $k\in\{\lfloor n/2\rfloor +1,\dots,n(n-1)/2\}$, we have $\calI_{k}^{n} = \emptyset$ and it follows from~\eqref{e:vanishing-a_I} and~\eqref{e:vanishing-b_I} that $f_{2k} =f_{2k+1} = 0$. 

We now prove that $f_{n+1}=0$ whenever $n$ is even. Assume that $n=2p$ with $p\geq 2$. Let $I=(\alpha_1\beta_1,\dots,\alpha_p\beta_p) \in \calI_{p}^{2p}$ be given. Note that  $\Spt (I) = \{1,\dots,2p\}$. Therefore, to show that $b_I=0$,  we need to verify that $b_I(e^{\alpha_j}) = b_I(e^{\beta_j}) = 0$ for every $j\in\{1,\dots,p\}$. Let $j\in\{1,\dots,p\}$ be given. Set $\omega_j:=\sum_{1\leq i \leq p,\, i\not= j} e^{\alpha_i\beta_i}$. On the one hand, since $f_{n+1} \in \affhn$, the function $t\in\R\mapsto f_{n+1}((e^{\alpha_j},\omega_j) \cdot (te^{\beta_j},0))$ is affine. On the other hand,
\begin{equation*}
\begin{split}
f_{n+1}((e^{\alpha_j},\omega_j) \cdot (t e_{\beta_j},0)) &= f_{n+1}(e^{\alpha_j} + t e^{\beta_j}, t e^{\alpha_j\beta_j} + \omega_j) \\
&= t b_I(e^{\alpha_j} + t e^{\beta_j}) = t b_I(e^{\alpha_j}) + t^2 b_I(e^{\beta_j})~.
\end{split}
\end{equation*}
Therefore the coefficient of $t^2$ vanishes, i.e., $b_I(e^{\beta_j})=0$. To prove that $b_I(e^{\alpha_j})=0$, we argue in a similar way considering the function $t\in\R\mapsto f_{n+1}((-e^{\beta_j},\omega_j) \cdot (t e^{\alpha_j},0))$.

All together we have shown that $f = \sum_{i=0}^{n} f_i$ with $f_i \in \affhni$ that can be written as in~(ii) when $i=2k$ is even, respectively, as in~(iii) when $i=2k+1$ is odd, which concludes the proof of the proposition.
\end{proof}

\begin{proposition} \label{prop:step2-free-case}
For $n \geq 3$, $i\in \{0,\dots,n\}$, $\nu \in \Lambda^n \R^n \setminus \{0\}$, every $f\in \affhni$  can be written as $f\nu =\varphi_\eta$ for some $\eta \in \Lambda^{n-i}\R^n$.
\end{proposition}

\begin{proof}
Assume with no loss of generality that $\nu = e^1 \wedge \cdots \wedge e^n$. We first prove the proposition when $i=2k$ is even. Let $f\in \affhn_{2k}$ and let $a_I \in \R$, $I\in \calI_{k}^{n}$, be given by Proposition~\ref{prop:step1-free-case}~(ii) so that
\begin{equation*}
f(\theta,\omega) = \sum_{I\in \calI_{k}^{n}} a_I \, \omega_I~.
\end{equation*}
For $k=0$ we get that $f$ is constant and the required conclusion clearly holds true. Next, let us consider the case $n=2p$ with $p\geq 2$ and $k=p$. For $\omega \in \Lambda^2\R^{2p}$ we have
\begin{equation*}
 \om^p =  p!\sum_{I\in \calI^{2p}_p }\sigma (I) \om_I \, \nu 
\end{equation*}
where, given $I=(\a_1\b_1,\dots,\a_p\b_p)\in\calI^{2p}_p$, $\sigma(I)$ denotes the signature of the permutation of $\{1,\dots,2p\}$ given by $(1,2,\dots, 2p) \mapsto (\a_1,\b_1,\dots,\a_p,\b_p)$. Therefore it suffices to prove that
\begin{equation} \label{e:n=2p-k=p}
\sigma(I) a_I = \sigma(I') a_{I'} \,\, \text{ for all } I,I' \in \calI^{2p}_p.
\end{equation}
On the one hand, since $f\in \affh(\free_{2p})$, we know that for all $(\theta,\omega) \in \free_{2p}$, $\theta'\in\Lambda^1\R^{2p}$, the function
\begin{equation*}
 t\in \R \mapsto f((\theta,\omega)\cdot(t\theta',0)) = \sum_{I\in\calI^{2p}_p} a_I \,(\om+t\theta\wedge\theta')_I
\end{equation*}
is affine. On the other hand, this function is a polynomial for which the coefficient of $t^2$ is given by
\begin{equation*}
 \sum_{I\in\calI^{2p}_p}  a_I \sum_{\substack{ H,K\in\calI^{2p}_1\\H, K\subset  I,\; H\neq K}} \om_{I\setminus(H\cup K)}\, (\theta\wedge\theta')_H \, (\theta\wedge\theta')_K =\sum_{ L\in\calI^{2p}_{p-2}  }\om_L \sum_{\substack{ I\in\calI^{2p}_p\\I\supset L}} a_I \, (\theta\wedge\theta')_{I\setminus L}
\end{equation*}
and must therefore vanish. Since this holds true for all $\om\in\Lambda^2\R^{2p}$, it follows that for all $L\in\calI^{2p}_{p-2}$, $\theta, \theta'\in\Lambda^1\R^{2p}$,
\begin{equation*} 
  \sum_{\substack{ I\in\calI^{2p}_p\\I\supset L}} a_I \, (\theta\wedge\theta')_{I\setminus L} =0.
\end{equation*}
Now let $L\in\calI^{2p}_{p-2}$ be given and let $1\leq \alpha < \beta < \delta < \gamma \leq 2p$ be such that $\{1,\dots,2p\} \setminus \Spt (L) = \{\alpha,\beta,\delta,\gamma\}$. Then the previous equality reads as
\begin{equation*}
\begin{split}
a_{L\cup (\alpha \beta,\delta \gamma)}(\theta\wedge\theta')_{\alpha \beta}& (\theta\wedge\theta')_{\delta \gamma} \\
&+ a_{L\cup (\alpha \delta, \beta \gamma)} (\theta\wedge\theta')_{\alpha \delta} (\theta\wedge\theta')_{\beta \gamma} + a_{L\cup (\alpha \gamma,\beta \delta)} (\theta\wedge\theta')_{\alpha \gamma} (\theta\wedge\theta')_{\beta \delta} =0
\end{split}
\end{equation*}
for all $\theta, \theta'\in\Lambda^1\R^{2p}$. Looking at the coefficient of $\theta_\alpha \theta'_\beta \theta_\delta \theta'_\gamma$ we get that $a_{L\cup (\alpha \beta,\delta \gamma)} = a_{L\cup (\alpha \gamma,\beta \delta)}$. Looking at the coefficient of $\theta_\alpha \theta'_\delta \theta_\beta \theta'_\gamma$ we get that $a_{L\cup (\alpha \delta, \beta \gamma)} = - a_{L\cup (\alpha \gamma,\beta \delta)}$. Therefore we have proved that $\sigma(I) a_I=\sigma(I') a_{I'}$ for all $I, I'\in\calI^{2p}_p$ such that $I, I'\supset L$ for some $L\in \calI^{2p}_{p-2}$. Since one can pass from any $I\in\calI^{2p}_p$ to any $I'\in\calI^{2p}_p$ by a finite number of such steps, \eqref{e:n=2p-k=p} follows.

Let us now consider the case $n\geq 3$ and $i=2k$ is even with $2\leq i \leq n-1$. For $J \in \calJ_{2k}^n$ set $\calI_{k}^J := \{I \in \calI_k^n: \Spt (I) =  \spt J   \}$  and define $f_J:\free_n\rightarrow\R$ by $f_J(\theta,\omega) := \sum_{I\in\calI_{k}^J} a_I\, \omega_I$ so that
\begin{equation*}
f= \sum_{J \in \calJ_{2k}^n} f_J~.
\end{equation*}
Set $\Lambda^1 J := \Span\{e^j: j\in \spt J\}$ and $\Lambda^2 J := \Span\{e^{jj'}: j,j'\in \spt J\}$. We have $f_J((\theta,\omega)\cdot(t\theta',0)) = f((\theta,\omega)\cdot(t\theta',0))$ for all $\theta,\theta'\in\Lambda^1 J$, $\omega\in\Lambda^2 J$. Since $f\in \affhn$, it follows that the restriction of $f_J$ to $\Lambda^1 J \times \Lambda^2 J \cong \free_{2k}$ belongs to $ \affh(\free_{2k})$. Since $(f_J \circ \delta_t)(\theta,\omega) = t^{2k}f_J(\theta,\omega)$ for all $(\theta,\omega) \in \Lambda^1 J \times \Lambda^2 J$ and all $t\in\R^*$, we get that the restriction of $f_J$ to $\Lambda^1 J \times \Lambda^2 J$ belongs to $\affh(\free_{2k})_{2k}$ and it follows from the previous case that there is $\eta_J\in\R$ such that $f_J(\theta,\omega)\,  e^J= \eta_J \, \omega^k$ for all $(\theta,\omega)\in \Lambda^1 J \times \Lambda^2 J$. Since $f_J$ does not depend on $\theta$, this equality holds actually true for all $\theta \in \Lambda^1\R^n$, $\omega\in \Lambda^2 J$. For $(\theta,\omega) \in \free_n$, we have $f_J(\theta,\omega) = f_J(\theta,\Pi_J(\omega))$ where $\Pi_J:\Lambda^2\R^n\rightarrow\Lambda^2 J$ denotes the projection map given by $\Pi_J(\omega):=\sum_{\substack{1\leq j<j'\leq n \\ j,j'\in\spt J  }}
 \om_{jj'}\, e^{jj'}$. Therefore, for $(\theta,\omega) \in \free_n$, we have
\begin{equation*}
f(\theta,\omega)\, \nu = \sum_{J \in \calJ_{2k}^n} f_J(\theta,\Pi_J(\omega))\, \nu = \sum_{J \in \calJ_{2k}^n} \sigma_J \, \eta_J \, \Pi_J(\omega)^k \wedge e^{J^c}
\end{equation*}
where $\sigma_J \in \{-1,1\}$ is such that $\nu = \sigma_J \, e^J \wedge e^{J^c}$. Now, note that $(\omega-\Pi_J(\omega))\wedge e^{J^c}=0$, therefore $\Pi_J(\omega)^k \wedge e^{J^c} = \omega^k \wedge e^{J^c}$, and the previous equality becomes 
\begin{equation*}
f(\theta,\omega)\, \nu = \sum_{J \in \calJ_{2k}^n} \sigma_J \, \eta_J \, \omega^k \wedge e^{J^c}  = \omega^k \wedge \eta
\end{equation*}
where $\eta := \sum_{J \in \calJ_{2k}^n} \sigma_J \, \eta_J \, e^{J^c}\in\Lambda^{n-2k}\R^n$, which concludes the proof of the proposition when $i=2k$ is even.

We now consider the case where $i=2k+1$ is odd. Let $f\in \affhn_{2k+1}$ and let $b_I:\Lambda^1\R^n \rightarrow \R$, $I\in \calI_{k}^{n}$, be linear forms given by Proposition~\ref{prop:step1-free-case}~(iii) so that
\begin{equation*} \label{e:case-2k+1}
f(\theta,\omega) = \sum_{I\in \calI_{k}^{n}} b_I(\theta) \, \omega_I~.
\end{equation*}
For $k=0$ we get that $f(\theta,\omega) = b_\emptyset (\theta)$ and the required conclusion clearly holds true. Thus assume that $k\in \{1,\dots,\lfloor (n-1)/2 \rfloor\} $ and let $\theta\in \Lambda^1\R^n$ be given. As in the proof of~\eqref{e:vanishing-b_I}, consider the function $g:\free_n \rightarrow \R$ given by $g(\tau,\omega):=f_{2k+1}(\theta,\omega)$ for $(\tau,\omega)\in \free_n$. We have $g\in \affhn$, see the proof of~\eqref{e:vanishing-b_I}, and since $g\circ \delta_t = t^{2k} g$, it follows that $g\in \affhn_{2k}$. Then we know from the previous cases and Corollary~\ref{cor:injectivity} that there is a unique $\overline{\eta}(\theta) \in \Lambda^{n-2k}\R^n$ such that $f(\theta,\omega)\nu = \omega^k \wedge \overline{\eta}(\theta)$ for all $\omega\in\Lambda^2\R^n$. Next, it follows from the linearity of the $b_I$ that the map $\theta\in\Lambda^1\R^n \mapsto  \omega^k \wedge \overline{\eta}(\theta)$ is linear for every $\omega\in\Lambda^2\R^n$. Since $\Span\{\omega^k: \omega \in \Lambda^2\R^n\} = \Lambda^{2k}\R^n$, see~\eqref{e:span-2k}, we get that $\overline{\eta}:\Lambda^1\R^n \rightarrow \Lambda^{n-2k}\R^n$ is linear. We now claim that $\theta \wedge \overline{\eta}(\theta) = 0$ for all $\theta \in \Lambda^1\R^n$. Indeed, on the one hand, we know that the function $t\in\R \mapsto f((\theta',\omega)\cdot(t \theta,0)) = f(\theta'+t \theta,\omega + t \theta'\wedge \theta)$ is affine for all $(\theta,\omega)\in\free_n$, $\theta'\in\Lambda^1\R^n$. On the other hand, we have
\begin{equation*}
\begin{split}
f(\theta'+t\theta,\omega + t\theta'\wedge \theta) \, \nu &= (\omega + t\theta'\wedge \theta)^k \wedge \overline{\eta}(\theta'+t\theta)\\
&=(\omega^k + kt \, \omega^{k-1}\wedge\theta'\wedge \theta) \wedge (\overline{\eta}(\theta') + t\,\overline{\eta}(\theta))\\
&= \omega^k \wedge \overline{\eta}(\theta') + t(\omega^k \wedge \overline{\eta}(\theta)+ k \, \omega^{k-1}\wedge\theta'\wedge \theta \wedge \overline{\eta}(\theta')) \\
&\phantom{=\,} + t^2 \,k\, \omega^{k-1}\wedge\theta'\wedge \theta \wedge \overline{\eta}(\theta)~,
\end{split}
\end{equation*}
and hence the coefficient of $t^2$ vanishes, i.e., $\omega^{k-1}\wedge\theta'\wedge \theta \wedge \overline{\eta}(\theta) = 0$ for every $\theta,\theta'\in\Lambda^1\R^n$, $\omega\in\Lambda^2\R^n$. Since $\Span\{\omega^{k-1}\wedge\theta': \theta'\in\Lambda^1\R^n, \omega \in \Lambda^2\R^n\} = \Lambda^{2k-1}\R^n$, see~\eqref{e:span-2k+1}, we get that $\theta\wedge\ol\eta(\theta)=0$ for all $\theta \in \Lambda^1\R^n$, as claimed. Then the required conclusion follows from Proposition~\ref{prop:eta-map}, and this concludes the proof of the proposition.
\end{proof}

\section{Horizontally affine functions on arbitrary step-2 Carnot algebras} \label{sect:h-affine-maps-arbitrary-case}

In this section we prove Theorem~\ref{thm:h-affine-maps-nonfree-case}, that will be deduced from Theorem~\ref{thm:h-affine-maps-free-case} writing a step-2 rank-$r$ Carnot algebra that is not isomorphic to $\free_r$ as a proper quotient of $\free_r$, and Theorem~\ref{thm:h-affine-maps-vs-surjective-Carnot-morphism}. The main argument for proving both theorems is given in Lemma~\ref{lem:factor-through-quotient} where we characterize those functions in $\affhn_j$ that factor through $\free_n / \fraki$ where $\fraki$ a graded ideal of $\free_n$. We refer to \eqref{e:anh} and \eqref{e:anhk} for the notions of annihilators.  For notational convenience, we shall again identify $\free_n$ with $\Lambda^1\R^n \times \Lambda^2\R^n$ throughout this section.

\begin{lemma} \label{lem:factor-through-quotient} Let $n\geq 2$ and $\fraki$  be a graded ideal of $\free_n$. For $j\in\{0,\dots,n\}$, $\eta \in \Lambda^{n-j} \R^n$, the map $\varphi_\eta$ given by~\eqref{e:def-phieta} factors through $\free_n / \fraki$ if and only if $\eta \in \anhk[n-j] \,\fraki$.
\end{lemma}

\begin{proof} Write the graded ideal $\fraki$ of $\free_n \cong \Lambda^1\R^n \times \Lambda^2\R^n$ as  $\fraki = \fraki_1 \times \fraki_2$ where $\fraki_1$, $\fraki_2$ are linear subspaces of $\Lambda^1 \R^n$, $\Lambda^2\R^n$ such that $\theta \wedge \theta' \in \fraki_2$ for all $\theta \in \Lambda^1\R^n$, $\theta'\in \fraki_1$. Recall that $\varphi: \free_n \rightarrow \Lambda^n \R^n$ factors through $\free_n / \fraki$ if and only if $\varphi(\theta+\tau, \omega + \zeta) = \varphi(\theta, \omega)$ for all $(\theta, \omega) \in \free_n $, $(\tau,\zeta) \in \fraki$. If $\eta \in \Lambda^n \R^n$ then $\varphi_\eta$ is  constant and therefore clearly factors through $\free_n / \fraki$. Since $\anhk[n] \,\fraki =\Lambda^n \R^n$, see~\eqref{e:anh-0}, this proves the lemma for $j=0$. If $\eta \in \Lambda^{n-1} \R^n$ then $\varphi_\eta(\theta,\omega) = \theta \wedge \eta$. Therefore $\varphi_\eta$ factors through $\free_n / \fraki$ if and only if $\tau\wedge \eta= 0$ for all $\tau \in \fraki_1$, i.e., $\eta \in \anhk[n-1] \,\fraki_1 = \anhk[n-1] \fraki$, where the last equality comes from~\eqref{e:anh-1}, which proves the lemma for $j=1$. Now let $j\in \{2,\dots,n\}$. For $\eta \in \anhk[n-j] \,\fraki$, it easily follows from~\eqref{e:def-phieta} that $\varphi_\eta$ factors through $\free_n / \fraki$. Conversely, let $\eta\in \Lambda^{n-j}\R^n$ and assume that $\varphi_\eta$ factors through $\free_n / \fraki$. Then, for all $(\theta,\omega)\in\free_n$, $\zeta \in \fraki_2$, $t\in \R$, we have $\varphi_\eta(\theta,\omega + t \zeta) = \varphi_\eta(\theta,\omega)$, i.e.,
\begin{equation*} 
\left\{
\begin{aligned}[2]
 &(\omega + t\zeta)^k \wedge \eta =\omega^k \wedge \eta & & \text{ if }  j=2k \text{ is even}
 \\
 &\theta \wedge (\omega + t\zeta)^k \wedge \eta  = \theta \wedge \omega^k \wedge \eta  & & \text{ if }  j=2k+1 \text{ is odd}.
 \end{aligned}\right.
\end{equation*}
Identifying the coefficient of degree 1 in $t$, we get that for all $\theta \in \Lambda^1\R^n$,  $\omega \in \Lambda^2\R^n$, $\zeta \in \fraki_2$,
\begin{equation*} 
\left\{
\begin{aligned}[2]
 &\omega^{k-1} \wedge \zeta \wedge \eta =0 & & \text{ if }  j=2k \text{ is even}
 \\
 &\theta \wedge \omega^{k-1} \wedge \zeta \wedge \eta =0 & & \text{ if }  j=2k+1 \text{ is odd}.
 \end{aligned}\right.
\end{equation*}
It then follows from~\eqref{e:span-2k} when $j$ is even,~\eqref{e:span-2k+1} when $j$ is odd, and Lemma~\ref{lem:annihilator} that $\zeta \wedge \eta \in \anhk[n-j+2] \Lambda^{j-2}\R^n = \{0\}$ for all $\zeta \in \fraki_2$, i.e., $\eta \in \anhk[n-j] \,\fraki_2 = \anhk[n-j] \,\fraki$, where the last equality comes from~\eqref{e:anh-i}, and this concludes the proof of the lemma.
\end{proof}

We first prove Theorem~\ref{thm:h-affine-maps-nonfree-case}.

\begin{proof}[Proof of Theorem~\ref{thm:h-affine-maps-nonfree-case}]
Let $\frakg$ be a step-2 rank-$r$ Carnot algebra that is not isomorphic to $\free_r$. Clearly, $\bigoplus_{i=0}^{r-1} \affh (\frakg)_i \subset \affh (\frakg)$. To prove the converse inclusion, let $\pi : \free_r \rightarrow \frakg$ be a surjective Carnot morphism and $\nu\in \Lambda^r \R^r \setminus \{0\}$ be fixed. Recall for further use that $\ker \pi$ is a non trivial graded ideal of $\free_r$. Let $f \in \affh (\frakg)$. Then $f\circ \pi \in \affhn$ by Lemma~\ref{lem:h-affine-maps-vs-Carnot-morphisms} and it follows from Theorem~\ref{thm:h-affine-maps-free-case} that there are $\eta_i \in \Lambda^{r-i} \R^r$, $i\in\{0,\dots,r\}$, such that $(f\circ \pi) \nu = \sum_{i=0}^r \varphi_{\eta_i}$. We claim that each $\varphi_{\eta_i}$ factors through $\free_r / \Ker \pi$. Indeed, since $\pi$ commutes with dilations, we have for all $(\theta,\omega) \in \free_r$, $(\tau,\zeta) \in \Ker\pi$, $t \in \R^*$,
\begin{equation*}
\begin{split}
\sum_{i=0}^r t^i \varphi_{\eta_i}(\theta + \tau, \omega + \zeta) &= \sum_{i=0}^r (\varphi_{\eta_i}\circ \delta_t)(\theta + \tau, \omega + \zeta)\\
& = (f \circ \pi \circ \delta_t) (\theta + \tau, \omega + \zeta) \, \nu  = (f \circ \delta_t \circ \pi) (\theta + \tau, \omega + \zeta) \, \nu\\
& = (f \circ \delta_t \circ \pi) (\theta, \omega ) \, \nu  = (f \circ \pi \circ \delta_t)(\theta, \omega ) \, \nu \\
& =\sum_{i=0}^r (\varphi_{\eta_i}\circ \delta_t)(\theta, \omega ) = \sum_{i=0}^r t^i \varphi_{\eta_i}(\theta, \omega)~.
\end{split}
\end{equation*}
This implies that for all $i\in\{0,\dots,r\}$, $\varphi_{\eta_i}(\theta + \tau, \omega + \zeta) = \varphi_{\eta_i}(\theta, \omega)$ for all $(\theta,\omega) \in \free_r$, $(\tau,\zeta) \in \Ker\pi$, i.e.,  $\varphi_{\eta_i}$ factors through $\free_r / \Ker \pi$, as claimed. Since $\pi$ is surjective, it follows that for each $i\in\{0,\dots,r\}$ there is $f_i : \frakg \rightarrow \R$ such that $(f_i \circ \pi) \nu = \varphi_{\eta_i}$. Since $\varphi_{\eta_i} \in \affh(\free_n ,\Lambda^n\R^n)_i$, we get from Corollary~\ref{cor:h-affine-maps-i-vs-Carnot-morphisms} that $f_i \in \affh (\frakg)_i$ (note indeed that the analogue of Corollary~\ref{cor:h-affine-maps-i-vs-Carnot-morphisms} holds true for $\Lambda^n\R^n$-valued functions). Let us now verify that $f_r = 0$. Since $\varphi_{\eta_r}$ factors through $\free_r / \Ker \pi$, we know from Lemma~\ref{lem:factor-through-quotient} that $\eta_r \in \anhk[0] \ker \pi$. Since $\frakg$ is not isomorphic to $\free_r$, we have $\ker \pi \not=\{0\}$ and hence $\anhk[0] \ker \pi =\{0\}$. Therefore $\eta_r =0$ and hence $f_r=0$. All together we get that $f= \sum_{i=0}^{r-1} f_i \in \bigoplus _{i=0}^{r-1} \affh (\frakg)_i$ and this concludes the proof of Theorem~\ref{thm:h-affine-maps-nonfree-case}.
\end{proof}

We now prove Theorem~\ref{thm:h-affine-maps-vs-surjective-Carnot-morphism}

\begin{proof}[Proof of Theorem~\ref{thm:h-affine-maps-vs-surjective-Carnot-morphism}]
Let $\frakg$ be a step-2 rank-$r$ Carnot algegra, $n\geq r$, and $\pi:\free_n \rightarrow \frakg$ be a surjective Carnot morphism. Let $i\in\{0,\cdots,n\}$, $\nu \in \Lambda^n \R^n \setminus \{0\}$, $\eta\in \anhk[n-i] \ker \pi$. Since $\ker\pi$ is a graded ideal of $\free_n$, we know from Lemma~\ref{lem:factor-through-quotient} that $\varphi_\eta$ factors through $\free_n / \ker \pi$ and since $\pi$ is surjective we get the existence of a unique function $f:\frakg \rightarrow \R$ such that $(f \circ \pi)\nu = \varphi_\eta$. Furthermore, since $\varphi_\eta \in \affh(\free_n ,\Lambda^n\R^n)_i$, we get from Corollary~\ref{cor:h-affine-maps-i-vs-Carnot-morphisms} that $f\in  \affh (\frakg)_i$, which concludes the proof of Theorem~\ref{thm:h-affine-maps-vs-surjective-Carnot-morphism}~(i). Conversely, let $f\in  \affh (\frakg)_i$. Then $f\circ \pi \in \affhni$ by Corollary~\ref{cor:h-affine-maps-i-vs-Carnot-morphisms} and it follows from Theorem~\ref{thm:h-affine-maps-free-case}~(ii) that there is a unique $\eta \in \Lambda^{n-i} \R^n$ such that $(f\circ \pi)\nu = \varphi_\eta$. This equality shows in turn that $\varphi_\eta$ factors through $\free_n / \ker \pi$ and hence $\eta\in \anhk[n-i] \ker \pi$ by Lemma~\ref{lem:factor-through-quotient}, which concludes the proof Theorem~\ref{thm:h-affine-maps-vs-surjective-Carnot-morphism}~(ii). By linearity of the map $\eta \mapsto \varphi_\eta$, we get that the bijective map $\eta \in \anhk[n-i] \ker \pi \mapsto f\in \affh (\frakg)_i$ where $f$ is given by Theorem~\ref{thm:h-affine-maps-vs-surjective-Carnot-morphism}~(i) is linear and therefore is an isomorphism of vector spaces, which concludes the proof of Theorem~\ref{thm:h-affine-maps-vs-surjective-Carnot-morphism}~(iii). By Theorem~\ref{thm:h-affine-maps-free-case}~(i) and Theorem~\ref{thm:h-affine-maps-nonfree-case}, it follows that $\affh (\frakg)$ and $\bigoplus_{i=0}^{n}\anhk[n-i] \ker \pi$ are isomorphic as vector spaces. Finally, since $\ker\pi$ is a graded ideal of $\free_n$, we get from Corollary~\ref{cor:anh-homogeneous-ideal} that $\anh \ker\pi = \bigoplus_{i=0}^{n}\anhk[n-i] \ker \pi$. Therefore $\affh (\frakg)$ and $\anh \ker\pi$ are isomorphic as vector spaces, which concludes the proof of Theorem~\ref{thm:h-affine-maps-vs-surjective-Carnot-morphism}~(iv). 
\end{proof}

\begin{remark} \label{rmk:upper-bound-dimension}
If $\frakg$ is a step-2 rank-$r$ Carnot algebra that is not isomorphic to $\free_r$ then $\dim \affh (\frakg) \leq 2^r - r+1$. Indeed consider a surjective Carnot morphism $\pi : \free_r \rightarrow \frakg$. Then $\ker \pi$ is a non trivial graded  ideal of $\free_r$ that is contained in $\Lambda^2\R^r$. Therefore $\anhk[r]  \ker \pi = \Lambda^{r} \R^r$ by~\eqref{e:anh-0}, $\anhk[r-1]  \ker \pi = \anhk[r-1]  \{0\} = \Lambda^{r-1} \R^r$ where the first equality follows from~\eqref{e:anh-1}, and $\anhk[r-i]  \ker \pi \varsubsetneq \Lambda^{r-i} \R^r$ for $i\in\{2,\dots,r-1\}$. This latter claim indeed follows from the inclusion $\ker\pi \subset \Lambda^2\R^r$ together with the fact that $\anhk[2] \Lambda^{r-i} \R^r = \{0\}$ for $i\in\{2,\dots,r-1\}$, see Lemma~\ref{lem:annihilator}. By Theorem~\ref{thm:h-affine-maps-vs-surjective-Carnot-morphism}~(iii) we get that $\affh (\frakg)_i$ and $\Lambda^{r-i} \R^r$ are isomorphic for $i\in\{0,1\}$ and  $\dim \affh (\frakg)_i  \leq \dim \Lambda^{r-i} \R^r -1$ for $i\in \{2,\dots,r-1\}$. Therefore $\dim \affh (\frakg) \leq 2^r - r+1$ by Theorem~\ref{thm:h-affine-maps-nonfree-case}.
\end{remark}

\begin{remark} \label{rmk:affhj>=i}
It follows from Theorems~\ref{thm:h-affine-maps-free-case},~\ref{thm:h-affine-maps-nonfree-case},~\ref{thm:h-affine-maps-vs-surjective-Carnot-morphism}~(iii), and Lemma~\ref{lem:anhk=0-anhj=0-j<=k} that if $\frakg$ is a step-2  Carnot algebra then  $\affh (\frakg)_i =\{0\}$ for some non negative integer $i$ if and only if $\bigoplus_{j\geq i} \affh (\frakg)_j =\{0\}$.
\end{remark}

\begin{remark} \label{rmk:trivial-anh-kernel} Note that it follows from Theorems~\ref{thm:h-affine-maps-free-case} and~\ref{thm:h-affine-maps-vs-surjective-Carnot-morphism}~(iii) that if $n>r\geq 2$ and $\pi:\free_n\rightarrow\free_r$ is a surjective Carnot morphism then $\anhk[n-i] \Ker\pi =\{0\}$ for $i\in \{r+1,\dots,n\}$. Similarly, it follows from Theorems~\ref{thm:h-affine-maps-nonfree-case} and~\ref{thm:h-affine-maps-vs-surjective-Carnot-morphism}~(iii) that if $\frakg$ is a step-2 rank-$r$ Carnot algebra that is not isomorphic to $\free_r$, $n\geq r$, and $\pi : \free_n \rightarrow \frakg$ is a surjective Carnot morphism then $\anhk[n-i] \Ker\pi =\{0\}$ for $i\in \{r,\dots,n\}$. 
\end{remark}

\section{Step-2 Carnot algebras where horizontally affine functions are affine} \label{sect:h-affine-maps-are-affine}

\subsection{Characterization} \label{subsect:h-affine-maps-are-affine}
This section is devoted to the proof of Theorem~\ref{thm:h-affine-maps-are-affine} that characterizes step-2 Carnot algebras where h-affine functions are affine. We begin with an easy consequence of Theorem~\ref{thm:h-affine-maps-vs-surjective-Carnot-morphism}.

\begin{lemma} \label{lem:affhi<=2-contained-in-aff}
Let $\frakg$ be a step-2 Carnot algebra. Then $\bigoplus_{i=0}^2 \affh (\frakg)_i\subset  \aff  (\frakg)$.
\end{lemma}

\begin{proof}
Let $f \in \bigoplus_{i=0}^2 \affh (\frakg)_i$. Let $r:=\rank\frakg$, $\pi : \free_r \rightarrow \frakg$ be a surjective Carnot morphism, and $\nu \in \Lambda^r\R^r\setminus\{0\}$. By Theorem~\ref{thm:h-affine-maps-vs-surjective-Carnot-morphism}~(ii) there are $\eta_i \in \anhk[n-i] \ker \pi$, $i\in\{0,1,2\}$, such that $(f\circ \pi) \nu = \sum_{i=0}^2 \varphi_{\eta_i}$. Then it clearly follows from the form of $\varphi_{\eta_i}$, see~\eqref{e:def-phieta}, together with the fact that $\pi$ is a surjective Carnot morphism that $f \in \aff (\frakg)$.
\end{proof}

We now turn to the proof of Theorem~\ref{thm:h-affine-maps-are-affine}.

\begin{proof}[Proof of Theorem \ref{thm:h-affine-maps-are-affine}] The equivalence between Theorem \ref{thm:h-affine-maps-are-affine}~(i) and (ii) follows from Lemmas \ref{lem:aff-contained-in-affhi<=2} and \ref{lem:affhi<=2-contained-in-aff}. Next, Theorem~\ref{thm:h-affine-maps-are-affine}~(ii) clearly implies Theorem~\ref{thm:h-affine-maps-are-affine}~(iii) since $\affh (\frakg)_i$ are linear subspaces of $\affh (\frakg)$ that are in direct sum, recalling also that $\bigoplus_{i\geq 3} \affh (\frakg)_i =\{0\}$ if and only if $\affh (\frakg)_3 =\{0\}$, see Remark~\ref{rmk:affhj>=i}.

Now, assume that $\affh(\frakg)_3 = \{0\}$. Let $b:\frakg_1 \times \frakg_2\rightarrow \R$ be a bilinear form such that $b(x,[x,x'])=0$ for all $x,x'\in\frakg_1$. Identifying $\frakg$ with $\frakg_1 \times \frakg_2$, we have for $x,x'\in \frakg_1$, $z\in \frakg_2$, $t\in \R$,
\begin{equation*}
\begin{split}
b((x,z)\cdot(tx',0))&=b(x+tx', z+t[x,x'])\\
& = b(x, z)+tb(x', z) + tb(x,[x,x']) + tb(x',[x,x'])\\
& = b(x, z)+tb(x', z)
\end{split}
\end{equation*}
and $b(\delta_t(x,z)) = b(tx,t^2z) = t^3 b(x,z)$. Therefore $b\in \affh(\frakg)_3$ and hence $b=0$, which proves that $\frakg$ is $\calI$-null. 

Next, assume that $\frakg$ is $\calI$-null. Let $n\geq \max\{3,r\}$ and $\pi : \free_n \rightarrow \frakg$ be a surjective Carnot morphism and let us verify that $\anhk[n-3] \ker \pi = \{0\}$. Let $\nu \in \Lambda^n \R^n \setminus \{0\}$ be given.  By Theorem~\ref{thm:h-affine-maps-vs-surjective-Carnot-morphism}~(i), for $\eta \in \anhk[n-3] \ker \pi$ there is $f\in \affh(\frakg)_3$ such that $(f\circ \pi) \nu = \varphi_\eta$. Identifying $\free_n$ with $\Lambda^1\R^n \times \Lambda^2\R^n$, we have $\varphi_\eta(\theta,\omega) = \theta \wedge \omega \wedge \eta$. Therefore $\varphi_\eta: \Lambda^1\R^n \times \Lambda^2 \R^n \rightarrow \Lambda^n\R^n$ is bilinear. Since $\pi$ is a surjective Carnot morphism, it follows that $f:\frakg \cong \frakg_1 \times \frakg_2 \rightarrow \R$ is bilinear as well. Furthermore, for $x,x' \in \frakg_1$, let $\theta ,\theta' \in \Lambda^1 \R^n$ be such that $x=\pi (\theta)$, $x'=\pi(\theta')$. Then $[x,x'] = \pi(\theta\wedge \theta')$ and hence, identifying $\frakg$ with $\frakg_1\times \frakg_2$, we have $(x,[x,x']) =\pi(\theta,\theta\wedge \theta')$. Therefore $f(x,[x,x'])\nu =\varphi_\eta(\theta, \theta \wedge \theta')  =0$. Since $\frakg$ is $\calI$-null, it follows that $f=0$ and hence $\eta=0$ by Corollary~\ref{cor:injectivity}. Therefore $\anhk[n-3] \ker \pi = \{0\}$, as wanted. Recall that this is in turn equivalent to $\bigoplus_{i=3}^n \anhk[n-i] \ker \pi = \{0\} $ by Lemma~\ref{lem:anhk=0-anhj=0-j<=k}. 

To conclude the proof of Theorem~\ref{thm:h-affine-maps-are-affine}, assume that there are $n\geq \max\{3,r\}$ and a surjective Carnot morphism $\pi:\free_n \rightarrow \frakg$ such that $\bigoplus_{i=3}^n \anhk[n-i] \ker \pi = \{0\}$. By Theorem~\ref{thm:h-affine-maps-vs-surjective-Carnot-morphism}~(iii) we get that $\bigoplus_{i=3}^n \affh (\frakg)_i =\{0\}$ and Theorems~\ref{thm:h-affine-maps-free-case}~(i) and~\ref{thm:h-affine-maps-nonfree-case} imply in turn Theorem~\ref{thm:h-affine-maps-are-affine}~(ii).
\end{proof}

\subsection{Examples}\label{exx} 

In this section we deduce from Theorem~\ref{thm:h-affine-maps-are-affine} sufficient conditions implying that h-affine functions are affine and necessary conditions that must be satisfied when this is the case. These conditions may be easier to verify on concrete examples than those given in the characterization obtained in Theorem~\ref{thm:h-affine-maps-are-affine}. We however illustrate with explicit examples to what extent some of these easier conditions cannot be turned into characterizations of step-2 Carnot algebras where h-affine functions are affine. We shall also see from some of these examples that, unlike affine functions, a h-affine function defined on a Lie subalgebra of a step-2 Carnot algebra may not admit a h-affine extension to the whole algebra.

\begin{proposition} \label{prop:I-null-if-adx-surjective}   
Let $\frakg$ be a step-2 Carnot algebra and assume there is $x\in \frakg_1$ such that $\ad_x:y\in \frakg_1 \mapsto [x,y] \in \frakg_2$ is surjective. Then $\affh(\frakg) = \aff (\frakg)$.
\end{proposition}

\begin{proof}
If $\ad_x$ is surjective for some $x \in \frakg_1$ then so is $\ad_{x'}$ for $x' \in U$ for some open neighborhood $U\subset \frakg_1$ of $x$. If $b:\frakg_1 \times \frakg_2 \rightarrow \R$ is a bilinear form such that $b(u,[u,v]) = 0$ for all $u,v \in \frakg_1$ then, by bilinearity of $b$ and surjectivity of $\ad_{x'}$, we get $b(x',z)=0$ for all $x'\in U$, $z\in \frakg_2$, and finally $b=0$, using once again the bilinearity of $b$ together with the fact that $U$ is a non empty open subset of $\frakg_1$. Therefore $\frakg$ is $\calI$-null and hence $\affh(\frakg) = \aff (\frakg)$ by Theorem~\ref{thm:h-affine-maps-are-affine}~(i)-(iv).
\end{proof}

Proposition~\ref{prop:I-null-if-adx-surjective} applies in particular to step-2 Carnot algebras of M\'etivier's type, i.e., step-2 Carnot algebras where $\ad_x$ is surjective for all $x\in \frakg_1 \setminus \{0\}$. 

The condition given in Proposition~\ref{prop:I-null-if-adx-surjective} about the surjectivity of $\ad_x$ for some $x\in \frakg_1$ implying that $\frakg$ is $\calI$-null should not be confused with the surjectivity of the Lie bracket $[\cdot,\cdot]: \frakg_1 \times \frakg_1 \rightarrow \frakg_2$. Indeed, step-2 Carnot algebras of M\'etivier's type are examples of $\calI$-null Lie algebras where the Lie bracket is surjective, whereas Example~\ref{ex:1} below gives an example of a $\calI$-null step-2 Carnot algebra  where the Lie bracket is not surjective. On the other hand, free step-2 Carnot algebras of rank 3 or higher are not $\calI$-null whereas the Lie bracket $(\theta, \theta') \in \Lambda^1 \R^n \times \Lambda^1\R^n \mapsto \theta \wedge \theta' \in \Lambda^2 \R^n$ is surjective if and only if $n=3$.

\begin{example} \label{ex:1} Let $\fraki$ be the graded ideal of $\free_4$ given by $\fraki:=\Span\{e^{12}+e^{34}\}$ and let $\frakg:= \free_4 / \fraki$. Elementary computations show that $\anh^1 \fraki = \{0\}$. Therefore $\affh(\frakg) = \aff (\frakg)$ and $\frakg$ is $\calI$-null by Theorem~\ref{thm:h-affine-maps-are-affine}~(i)-(iv)-(v). To see that the Lie bracket is not surjective we identify $\frakg$ with $\Lambda^1\R^4 \oplus \frakg_2$ where  $\frakg_2 :=\{\omega\in \Lambda^2 \R^4:\, \omega_{34} = 0\}$. The only non trivial bracket relations are given by 
\begin{equation*}
[\theta,\theta'] = ((\theta_1 \theta_2' - \theta_1' \theta_2) - (\theta_3 \theta_4' - \theta_3' \theta_4)) \, e^{12} + \sum_{13\leq ij \leq 24} (\theta_i \theta_j' - \theta_i' \theta_j) \,e^{ij}
\end{equation*}
for $\theta, \theta' \in \Lambda^1\R^4$ and we claim that 
\begin{equation*} \label{e:not-onto}
\operatorname{Im} [\cdot,\cdot] \cap \{ \omega \in \frakg_2: \omega_{14} = \omega_{23} = 0\} \subset \{\omega \in\frakg_2 : \omega_{12}^2 + 4\, \omega_{13} \, \omega_{24} \geq 0\}.
\end{equation*}
Indeed, let $\omega  = [\theta,\theta'] $ with $\theta, \theta' \in \Lambda^1\R^4$. Set $u_i:=(\theta_i,\theta_i') \in \R^2$ for $i=1,2,3,4$. We have $\omega_{ij} = \det (u_i,u_j)$ for $13\leq ij \leq 24$. Assume that $\omega_{14} = \omega_{23} = 0$, or equivalenlty, that $u_1$ and $u_4$, respectively $u_2$ and $u_3$, are colinear. Since $\omega_{12} = \det(u_1,u_2) - \det(u_3,u_4)$, we get that there are $s, t \in \R$ and $i_1\in \{1,4\}$, $i_2\in \{2,3\}$ such that $\omega_{12}^2 + 4\, \omega_{13} \, \omega_{24} = ((s+t)^2 -4st) \det(u_{i_1},u_{i_2})^2 = (s-t)^2 \det(u_{i_1},u_{i_2})^2 \geq 0$, as claimed.
\end{example}

Given a step-2 Carnot algebra $\frakg$, we say that $\frakg'$ is a quotient of $\frakg$ if $\frakg'$ is a step-2 Carnot algebra and there is a surjective Carnot morphism $\pi : \frakg \rightarrow \frakg'$. 

\begin{proposition} \label{prop:quotient-Inull-Lie-algebra}
Let $\frakg$ be a step-2 Carnot algebra such that $\affh(\frakg) = \aff (\frakg)$. Then $\affh(\frakg') = \aff(\frakg')$ for every quotient $\frakg'$ of $\frakg$.
\end{proposition} 

\begin{proof}
By~\cite[Lemma~2.3]{MR3086803} every quotient of a $\calI$-null Lie algebra is $\calI$-null and hence the proposition follows from Theorem~\ref{thm:h-affine-maps-are-affine}~(i)-(iv).
\end{proof}

Note that it may happen that $\aff(\frakg) \varsubsetneq \affh(\frakg)$ while $\affh(\frakg') = \aff(\frakg')$ for every proper quotient $\frakg'$ of $\frakg$, i.e., for every quotient of $\frakg$ that is not isomorphic to $\frakg$. A simple example is given by the free step-2 rank-3 Carnot algebra $\free_3$. Indeed, we know from Theorem~\ref{thm:h-affine-maps-free-case} that $\aff(\free_3) \varsubsetneq \affh(\free_3)$, whereas every proper quotient $\frakg'$ of $\free_3$ is either isomorphic to $\free_2$ or has rank 3 and is not isomorphic to $\free_3$, therefore $\affh(\frakg') = \aff(\frakg')$ by Theorems \ref{thm:heisenberg}, \ref{thm:h-affine-maps-nonfree-case}, and \ref{thm:h-affine-maps-are-affine}~(i)-(ii). See also~Example~\ref{ex:2} for another example that is not isomorphic to $\free_3$.

Note also that since $\aff(\free_3) \varsubsetneq \affh(\free_3)$, Proposition~\ref{prop:quotient-Inull-Lie-algebra} has the following immediate corollary.

\begin{corollary} \label{cor:G-with-F3-as-quotient}
Let $\frakg$ be a step-2 Carnot algebra that has $\free_3$ as one of its quotients. Then $\aff(\frakg) \varsubsetneq \affh(\frakg)$.
\end{corollary}

It may happen that $\aff(\frakg) \varsubsetneq \affh(\frakg)$ while $\frakg$ does not have $\free_3$ as one of its quotients, as shows Example~\ref{ex:2} where $\frakg$ is a step-2 rank-5 Carnot algebra such that $\aff(\frakg) \varsubsetneq \affh(\frakg)$ and $\affh(\frakg') = \aff(\frakg')$ for every proper quotient $\frakg'$ of $\frakg$.

\begin{example} \label{ex:2} 
Let $\fraki$ be the graded ideal of $\free_5$ given by $\fraki:=\anhk[2] \{\zeta\}$ where $\zeta:=e^{12} + e^{45}$ and let $\frakg:=\free_5 / \fraki$ so that $\frakg$ is in particular a step-2 rank-5 Carnot algebra and therefore is not isomorphic to $\free_3$. We have $\zeta \in \anhk[2] \fraki$ and hence $\aff (\frakg) \varsubsetneq \affh (\frakg)$ by Theorem~\ref{thm:h-affine-maps-are-affine}~(i)-(v). We now claim that $\affh(\frakg') = \aff(\frakg')$ for every proper quotient $\frakg'$ of $\frakg$. Indeed, let $\frakg'$ be a proper quotient of $\frakg$ and let $\pi:\frakg \rightarrow \frakg'$ be a surjective Carnot morphism.  Let $\overline{\pi}: \free_5 \rightarrow \frakg$ denote the quotient map so that $\pi \circ \overline{\pi}:\free_5 \rightarrow \frakg'$ is a surjective Carnot morphism and let us verify that $\anhk[2] \ker \pi \circ \overline{\pi} =\{0\}$. Let $\eta\in \anhk[2] \ker \pi \circ \overline{\pi}$. We have $\anhk[1] \{\zeta\} \oplus \anhk[2] \{\zeta\} = \{0\} \oplus \anhk[2] \{\zeta\} = \ker \overline{\pi} \subset \ker \pi \circ \overline{\pi} \subset \anhk[1] \{\eta\} \oplus \anhk[2] \{\eta\}$. Since $\anhk[2] \{\zeta\} =\Span\{e^{12} - e^{45}, e^{14}, e^{15}, e^{24},e^{25}\}$, we have $(e^{12} - e^{45}) \wedge \eta  =0$, implying $\eta_{34} = \eta_{35} = \eta_{45} - \eta_{12} = \eta_{13} = \eta_{23} = 0$, and $e^{ij}\wedge \eta  =0$ for $i=1,2$, $j=4,5$, implying $\eta_{ij} = 0$ for $i=1,2$, $j=4,5$. Therefore $\eta \in \Span\{\zeta\}$. If $\eta\not=0$ then $\anhk[1]\{\eta\} \oplus \anhk[2] \{\eta\} = \anhk[1] \{\zeta\} \oplus \anhk[2] \{\zeta\}$ which implies in turn $\ker \overline{\pi} = \ker \pi \circ \overline{\pi}$ and therefore $\ker \pi =\{0\}$. This contradicts the fact that $\frakg'$ is a proper quotient of $\frakg$ and hence $\eta=0$. Therefore $\anhk[2] \ker \pi \circ \overline{\pi} =\{0\}$ and it follows from Theorem~\ref{thm:h-affine-maps-are-affine}~(i)-(v) that $\affh(\frakg') = \aff(\frakg')$, as claimed. 
\end{example}

We recall that the direct product $\frakg \times \R^d$ of a step-2 Carnot algebra with an abelian Lie algebra and the direct product $\frakg \times \frakg'$ of step-2 Carnot algebras inherit naturally of a structure of step-2 Carnot algebra from those of $\frakg$ and $\frakg'$.

\begin{proposition} \label{prop:direct-product-Inull-Lie-algebras}
Let $\frakg, \frakg'$ be a step-2 Carnot algebras and $d\geq 1$ be an integer. Then the following holds true:
\begin{enumerate}
\item[(i)] $\affh(\frakg \times \frakg') = \aff(\frakg \times \frakg')$ if and only if $\affh(\frakg) = \aff(\frakg)$ and $\affh(\frakg') = \aff(\frakg')$
\smallskip
\item[(ii)] $\affh(\frakg \times \R^d) = \aff(\frakg \times \R^d)$ if and only if $\affh(\frakg) = \aff(\frakg)$.
\end{enumerate}
\end{proposition}

\begin{proof}
By~\cite[Lemma~2.3]{MR3086803} any finite direct product of $\calI$-null Lie algebras is $\calI$-null. Therefore it follows from Theorem~\ref{thm:h-affine-maps-are-affine}~(i)-(iv) that $\affh(\frakg \times \frakg') = \aff(\frakg \times \frakg')$ whenever $\affh(\frakg) = \aff(\frakg)$ and $\affh(\frakg') = \aff(\frakg')$. Similarly, since abelian Lie algebras are $\calI$-null (see Definition~\ref{def:I-null}), we have $\affh(\frakg \times \R^d) = \aff(\frakg \times \R^d)$ whenever  $\affh(\frakg) = \aff(\frakg)$. The converse implications in (i) and (ii) follow from Proposition~\ref{prop:quotient-Inull-Lie-algebra} noting that the projection maps from $\frakg \times \frakg'$ onto either $\frakg$ or $\frakg'$ and from  $\frakg \times \R^d$ onto $\frakg$ are surjective Carnot morphisms. 
\end{proof}

The next proposition is another simple consequence of Theorem~\ref{thm:h-affine-maps-are-affine} that gives a sufficient condition ensuring that h-affine maps are affine.

\begin{proposition} \label{prop:F3-as-a-subalgebra-when-not-Inull} Let $\frakg$ be a step-2 Carnot algebra such that $\aff(\frakg) \varsubsetneq \affh(\frakg)$. Then there is a Lie subalgebra of $\frakg$ isomorphic to $\free_3$.
\end{proposition}

\begin{proof} 
By Theorem~\ref{thm:h-affine-maps-are-affine}~(i)-(iv) we know that $\frakg$ is not $\calI$-null. Then let $b:\frakg_1\times \frakg_2 \rightarrow \R$ be a non zero bilinear form such that $b(x,[x,x']) = 0$ for all $x,x'\in \frakg_1$. Since $b\not=0$, there are $x\in \frakg_1$, $z\in \frakg_2$ such that $b(x,z)\not=0$. By bilinearity together with the fact that $[\frakg_1,\frakg_1] = \frakg_2$, it follows that there are $x_1,x_2,x_3\in \frakg_1$ such that $b(x_1,[x_2,x_3])=1$. We claim that $[x_1,x_2], [x_1,x_3], [x_2,x_3]$ are linearly independent and therefore the Lie subalgebra of $\frakg$ generated by $x_1,x_2,x_3$ is isomorphic to $\free_3$. Indeed, note that since $b(x,[x,x']) = 0$ for all $x,x'\in \frakg_1$, the trilinear form $(x, x', x'') \in \frakg_1 \times \frakg_1\times \frakg_1 \mapsto b(x,[x',x''])$ is alternating and therefore $b(x_3,[x_1,x_2])=- b(x_2,[x_1,x_3])=b(x_1,[x_2,x_3])=1$. Now let $s_1,s_2,s_3 \in \R$ be such that $s_3 [x_1,x_2] + s_2 [x_1,x_3] + s_1 [x_2,x_3] = 0$. For $i \in \{1,2,3\}$, we have $b(x_i, s_3 [x_1,x_2] + s_2 [x_1,x_3] + s_1 [x_2,x_3]) = s_i\, b(x_i,[x_k,x_l]) = 0$ where $k<l $ and $ \{k,l\}=  \{1,2,3\}\setminus\{i\}$. Since $b(x_i,[x_k,x_l])\not=0$ for such indices $i,k,l$, it follows that $s_1=s_2=s_3=0$, which concludes the proof of the lemma.
\end{proof}

We stress that $\calI$-null step-2 Carnot algebras may have Lie subalgebras isomorphic to $\free_3$, as shown in the following two examples.

\begin{example} \label{ex:3} [The quaternionic Heisenberg algebra.] Let $i,j,k$ denote the quaternion units satisfying $i^2=j^2=k^2=ijk=-1$ and denote by $\H:=\{q_1+iq_2+jq_3+kq_4:\, q_i \in \R\}$ the set of quaternions. Given $q= q_1+iq_2+jq_3+kq_4 \in \H$, denote by $\operatorname{Im}q:= iq_2+jq_3+kq_4$ its imaginary part and $\overline{q}:=q_1-iq_2-jq_3-kq_4$ its conjugate. Equip $\H\oplus\operatorname{Im}\H$ with the Lie bracket for which the only non trivial relations are given by $[q,q']:= \operatorname{Im}(\overline{q}\,q')$ for $q,q'\in \H$ which makes $\H\oplus\operatorname{Im}\H$ a step-2 Carnot algebra that is well known to be of Heisenberg type, and therefore of M\'etivier's type (see for instance~\cite{Kaplan80}). Therefore we have $\affh(\H\oplus\operatorname{Im}\H) = \aff(\H\oplus\operatorname{Im}\H)$ and $H\oplus\operatorname{Im}\H$ is $\calI$-null by Proposition~\ref{prop:I-null-if-adx-surjective}. We now claim that the Lie subalgebra of $\H\oplus\operatorname{Im}\H$ generated by any three linearly independent elements in $\H$ is isomorphic to $\free_3$. Indeed, for $q,q'\in \H$, we have $[q,q']=0$ if and only if $q$ and $q'$ are colinear. This indeed follows from the fact that for $q\in \H\setminus\{0\}$, the linear map $\ad_q: q' \in \H \mapsto [q,q'] \in \operatorname{Im}\H$ is surjective with $q\in \Ker \ad_q$, together with the fact that $\dim \H = 4$ and $\dim \operatorname{Im}\H = 3$. Then let $q_1, q_2, q_3\in \H$ be linearly independent and assume by contradiction that $\dim\Span\{[q_i,q_j]: i,j \in \{1,2,3\}\} \leq 2$. Exchanging the role of  $q_1, q_2, q_3$ if necessary, there are $s,t\in \R$ such that $[q_2,q_3] = s [q_1,q_2] + t [q_1,q_3]$. Then  $[q_2 - t q_1, sq_1 + q_3]=0$ which implies that $q_2 - t q_1$ and $sq_1 + q_3$ are colinear and contradicts the fact that $q_1, q_2, q_3$ are linearly independent. Therefore $\dim\Span\{[q_i,q_j]: i,j \in \{1,2,3\}\} = 3$ and the Lie subalgebra generated by $q_1,q_2,q_3$ is isomorphic to $\F_3$, as claimed.
\end{example}

\begin{example} \label{ex:4} Let $\frakg:= \free_4 / \fraki$ where $\fraki:=\Span\{e^{12}+e^{34}\}$ be the $\calI$-null step-2 Carnot algebra given by Example~\ref{ex:1}. We claim that the Lie subalgebra of $\frakg$ generated by any three linearly independent elements in $\Lambda^1\R^4$ is isomorphic to $\free_3$. Indeed let $\pi:\free_4 \rightarrow \frakg$ denote the quotient map. Let $\theta_1,\theta_2,\theta_3\in \Lambda^1\R^4$ be linearly independent and let $a,b,c\in \R$ be such that $a\,\pi(\theta_1 \wedge \theta_2) + b\, \pi(\theta_1 \wedge \theta_3) + c\, \pi(\theta_2 \wedge \theta_3) =0$, i.e., $a\,\theta_1 \wedge \theta_2 + b\, \theta_1 \wedge \theta_3 + c\, \theta_2 \wedge \theta_3 \in \Ker\pi$. We have $(a\,\theta_1 \wedge \theta_2 + b\, \theta_1 \wedge \theta_3 + c\, \theta_2 \wedge \theta_3)^2 = 0$ while $\Ker\pi=\Span\{e^{12}+e^{34}\}$ with $(e^{12}+e^{34})^2\not= 0$. It follows that $a\,\theta_1 \wedge \theta_2 + b\, \theta_1 \wedge \theta_3 + c\, \theta_2 \wedge \theta_3=0$ and hence $a=b=c=0$. This proves that $\dim \Span\{\pi(\theta_i \wedge \theta_j): \, i,j \in \{1,2,3\}\} =3$ and therefore the Lie subalgebra of $\frakg$ generated by $\theta_1,\theta_2,\theta_3$ is isomorphic to $\free_3$, as claimed.
\end{example}

Theorem~\ref{thm:h-affine-maps-are-affine}~(i)-(iv) together with Proposition~\ref{prop:F3-as-a-subalgebra-when-not-Inull} has the following immediate consequence.

\begin{proposition} Let $\frakg$ be a step-2 Carnot algebra with $\dim \frakg_2 \leq 2$. Then $\affh(\frakg) = \aff(\frakg)$.
\end{proposition}

If $\rank \frakg = \dim \frakg_2 = 3$ then $\frakg$ is isomorphic to $\free_3$ and therefore  $\aff(\frakg) \varsubsetneq \affh(\frakg)$. This fact generalizes to higher rank step-2 Carnot algebras $\frakg$ with $\dim \frakg_2 = 3$ in the following way.

\begin{proposition}
Let $\frakg$ be a step-2 Carnot algebra with $\dim \frakg_2 = 3$. Then $\affh(\frakg) = \aff(\frakg)$ if and only if $\frakg$ is not isomorphic to  $\free_3 \times \R^d$ for some non negative integer $d$. 
\end{proposition}

\begin{proof}
If $\frakg$ is isomorphic to the direct product $\free_3 \times \R^d$ for some non negative integer $d$ then $\frakg$ has $\free_3$ as one of its quotients and we know by Corollary~\ref{cor:G-with-F3-as-quotient} that $\aff(\frakg) \varsubsetneq \affh(\frakg)$. Conversely, assume that $\aff(\frakg) \varsubsetneq \affh(\frakg)$. First, if $\rank \G = 3$, since $\dim \frakg_2 = 3$, then $\frakg$ is isomorphic to $\free_3$. Next, assume that $r:=\rank \G \geq 4$. By Proposition~\ref{prop:F3-as-a-subalgebra-when-not-Inull}, there are $x_1,x_2,x_3\in \frakg_1$ generating a Lie subalgebra of $\frakg$ isomorphic to $\free_3$ and there is a bilinear form $b:\frakg_1\times \frakg_2 \rightarrow \R$ such that $b(x,[x,x']) = 0$ for all $x,x'\in \frakg_1$ and  $b(x_1,[x_2,x_3]) =1$. Set $V:=\Span\{x_1,x_2,x_3\}$. We claim that for every $x\in \frakg_1 \setminus V$ there is $x'\in V$ such that $x+x'$ lies in the center of $\frakg$, i.e., $[y,x+x'] = 0$ for all $y\in \frakg_1$. To prove this claim, let $x\in \frakg_1 \setminus V$ be given. Since $\dim \frakg_2= 3=\dim \Span \{[x_i, x_j]: 1\leq i<j\leq 3\}$, there are $x_i^j \in \R$, $i,j \in \{1,2,3\}$, such that $[x_j,x] = x_3^j \, [x_1,x_2] - x_2^j \, [x_1,x_3] + x_1^j \, [x_2,x_3]$ for $j \in \{1,2,3\}$. Set $x':= x_3^2\, x_1 + x_1^3\, x_2 + x_2^1 \, x_3$ and let us verify that $x+x' \in \operatorname{Center} (\frakg)$. We first verify that $[x_j,x+x']=0$ for $j \in \{1,2,3\}$. Since  $b(x,[x,x']) = 0$ for all $x,x'\in \frakg_1$, the bilinear form $(y,y') \in V\times V \mapsto b(y,[y',x])$ is skew-symmetric. For $y=\sum_{1\leq i \leq 3} \a_i\, x_i \in V$ and $y'=\sum_{1\leq j \leq 3} \b_j\, x_j\in V$, we have $b(y,[y',x]) = \sum_{1\leq i,j \leq 3} x_i^j\, \a_i\, \b_j$. By skew-symmetry, we get $x_i^j = -x_j^i$ for $i,j \in \{1,2,3\}$ and it follows that $[x_j,x+x']=0$ for $j \in \{1,2,3\}$. Now let $y\in \frakg_1$ and  write $[y,x+x'] = s_3 \,[x_1,x_2] - s_2 \, [x_1,x_3] + s_1 \, [x_2,x_3]$ with $s_i\in\R$. Since the trilinear form $b(\cdot,[\cdot,\cdot])$ is alternating, we have $b(x_i,[y,x+x'])=s_i$. On the other hand, $b(x_i,[y,x+x']) = - b(y,[x_i,x+x']) = 0$. Hence $s_i=0$ for $i\in\{1,2,3\}$ and therefore $[y,x+x'] = 0$, as wanted. It now clearly follows from this claim that one can complete $(x_1,x_2,x_3)$ into a basis $(x_1,\dots,x_r)$ of $\frakg_1$ in such a way that $x_4,\dots,x_r \in \operatorname{Center} (\frakg)$, i.e., $\frakg$ is isomorphic to $\free_3 \times \R^{r-3}$.
\end{proof}

To conclude this section, let us remark that a h-affine function defined on a Lie subalgebra of a step-2 Carnot algebra may not admit a h-affine extension to the whole algebra. Indeed, let $\frakg$ be a step-2 Carnot algebra such that $\affh(\frakg) = \aff(\frakg)$ and such that there is a Lie subalgebra $\frakf$ of $\frakg$ isomorphic to $\free_3$, see Examples~\ref{ex:3} and~\ref{ex:4}. Then $\aff(\frakf)  \varsubsetneq \affh(\frakf)$ and therefore there is $h\in \affh(\frakf) \setminus \aff(\frakf)$. Assume there is $\overline{h} \in \affh(\frakg)$ whose restriction to $\frakf$ is $h$. By assumption on $\frakg$, we have $\overline{h} \in \aff(\frakg)$. Since $\frakf$ is a linear subspace of $\frakg$, it follows that the restriction of $\overline{h}$ to $\frakf$ is affine, i.e., $h\in \aff(\frakf)$, which gives a contradiction. Recall that on the contrary an affine function defined on an affine subspace of a vector space can always be extended to an affine function on the whole space.

\section{Appendix about linear and exterior algebra} \label{sect:appendix-algebra}

We gather in this appendix some basic facts about linear and exterior algebra not pertaining to h-affine functions that have been used in the previous sections. 

We start with a characterization of affine maps between real vector spaces whose elementary proof is left to the reader.

\begin{proposition} \label{prop:affine-maps} Let $E$, $F$ be real vector spaces. A map $f:E\rightarrow F$ is affine if and only if for every $x,y\in E$, the map $t\in\R \mapsto f(x+ty) $ is affine.   
\end{proposition}

For the sake of completeness, we state below an elementary property of multiaffine functions that has been used in the proof of Proposition~\ref{prop:step1-free-case}. The proof can easily be done by induction on the dimension and is left to the reader.

\begin{proposition} \label{prop:homogeneous-polynomial}
Let $p\geq 1$ be an integer and $E$ be a $p$-dimensional real vector space. Let $f:E \rightarrow \R$ and assume that there is a basis $(e_1,\dots,e_p)$ of $E$ such that the map $t\in\R \mapsto f(v+te_j)$ is affine for every $v\in E$ and $j\in\{1,\dots,p\}$. Then $f$ is a linear combination of the $u_J$, where $J$ ranges over the subsets of $\{1,\dots,p\}$ and $u_J(v) := \prod_{j\in J} v_j$ for $v=\sum_{j=1}^p v_j  e_j$ with the convention $u_\emptyset(v):=1$.
\end{proposition}

The rest of this appendix is devoted to (basic) facts about exterior algebra that have been used throughout this paper. Although some of them look quite elementary, we were however unable to find references in the literature and thus provide proofs for the reader's convenience.

Given integers $n\geq 1$ and $k\geq 1$ we denote by $\Lambda^k \R^n$ the set of alternating $k$-multilinear forms over $\R^n$. For $k=0$ we set $\Lambda^0\R^n := \R$. We denote by $\Lambda^* \R^n := \bigoplus_{k\geq 0} \Lambda^k \R^n$ the exterior algebra equipped with exterior product $\wedge$. We recall that $\Lambda^k \R^n = \{0\}$ if $k>n$. For $\zeta \in \Lambda^* \R^n$ we set $\zeta^0 := 1$ and $\zeta^k := \underbrace{\zeta \wedge \cdots \wedge \zeta}_{k\, \text{times}}$ for $k\geq 1$.

\begin{lemma} \label{lem:span}
For $n\geq 1$,  we have in $\Lambda^*\R^n$
\begin{gather}
 \Span \{\omega^k: \, \omega\in\Lambda^2\R^n \}=\Lambda^{2k}\R^n~, \label{e:span-2k} \\
\Span \{\theta \wedge \omega^k: \,\theta\in\Lambda^1\R^n, \,\omega \in \Lambda^2\R^n \} = \Lambda^{2k+1}\R^n \label{e:span-2k+1} 
\end{gather}
for all $k\geq 0$.
\end{lemma}

\begin{proof}
We only need to consider the non trivial cases where $n\geq 2$ and $1\leq k\leq\lfloor n/2 \rfloor$. Then~\eqref{e:span-2k} and~\eqref{e:span-2k+1} follow from the identity 
\begin{equation}\label{e:span-2k-identity} 
 (\theta_{1}\wedge\theta_{2} + \cdots + \theta_{2k-1}\wedge \theta_{2k})^k =k!\, \theta_1\wedge\cdots\wedge\theta_{2k}
\end{equation}
for all $\theta_1,\dots,\theta_{2k} \in \Lambda^1\R^n$.
\end{proof}

Given $n\geq 1$  we set $\calJ_{0}^{n}:=\{\emptyset\}$, 
\begin{equation*}
\calJ_{k}^{n}:=\{(j_1,\dots,j_k)\in \N^k: 1\leq j_1 < \cdots < j_k\leq n\}
\end{equation*}
for $k\in \{1,\dots,n\}$, and $\calJ^n:=\cup_{0\leq k \leq n} \calJ_{k}^{n}$. We write $\spt \emptyset :=\emptyset$ and $\spt J := \{j_1,\dots,j_k\} \subset \N$ for $J=(j_1,\dots,j_k) \in \calJ_{k}^{n}$. Given $J,J' \in \calJ^n$ we denote by $J\setminus J'$ the unique element in $\calJ^{n}$ such that $\spt (J\setminus J') = \spt J \setminus \spt J'$ and we set  $J^c:= (1,\dots,n) \setminus J$. We fix a basis $(e_1,\dots,e_n)$ of $\R^n$ and denote by $(e^1,\dots,e^n)$ its dual basis. For $J=(j_1,\dots,j_k) \in \calJ_{k}^{n}$, we set $e^J:= e^{j_1} \wedge \cdots \wedge e^{j_k} \in \Lambda^k\R^n$ with the convention $e^{\emptyset} := 1$. 

The space of exterior annihilators of an element in $\Lambda^i\R^n$ of some given order $k$ has been introduced in~\cite{MR2847348}, see also~\cite[Section~2.2]{Csato_Dacorogna_Kneuss_12}. More generally, given $n\geq 1$, $k\geq 0$, and $A\subset \Lambda^* \R^n$, we define the annihilator of $A$  in $\Lambda^* \R^n$, respectively $\Lambda^k\R^n$, as
\begin{align}
\anh A  &:= \{\eta\in \Lambda^* \R^n:\, \eta \wedge \zeta = 0 \text{ for all } \zeta \in A\}~, \label{e:anh}\\
\anhk[k] A &:= \Lambda^k\R^n \cap \anh A~. \label{e:anhk}
\end{align}
We also set $\Lambda^{\geq k} \R^n := \bigoplus_{i\geq k} \Lambda^i \R^n$.

\begin{lemma} \label{lem:annihilator} For $n\geq 1$, $k\in \{0,\dots,n\}$, we have $\anh \Lambda^k \R^n = \Lambda^{\geq n-k+1} \R^n$.
\end{lemma}

\begin{proof}
For $k=0$ we clearly have $\anh \Lambda^0 \R^n = \{0\} = \Lambda^{\geq n+1} \R^n$. Let $n\geq 1$ and $k\in\{1,\dots,n\}$ be given. Clearly $\Lambda^{\geq n-k+1} \R^n \subset \anh \Lambda^k \R^n$. Since $\anh \Lambda^k \R^n$ is graded, i.e., $\anh \Lambda^k \R^n = \bigoplus_{i\geq 0} \anhk[i] \Lambda^k \R^n$ (see Lemma~\ref{lem:anh-subsets-of-Lambdak}), if equality fails then $\anhk[i] \Lambda^k \R^n \not=\{0\}$ for some $i \in \{0,\dots, n-k\}$. Let $\eta \in \anhk[i] \Lambda^k \R^n \setminus \{0\}$ and $J=(j_1,\dots,j_i) \in \calJ_{i}^{n}$ be such that $\eta(e_{j_1}, \dots, e_{j_i})\not=0$. Then $\eta \wedge e^{J^c} = \eta(e_{j_1}, \dots, e_{j_i})\, e^J \wedge e^{J^c} \not=0$. However, we can write $e^{J^c} = e^{J_1} \wedge e^{J_2}$ for some $J_1 \in  \calJ_{k}^{n}$, $J_2\in\calJ_{n-i-k}^{n}$, and since $\eta \in \anh \Lambda^k \R^n$, we get $\eta \wedge e^{J^c} = 0$, which gives a contradiction.
\end{proof}

Writing $\calF(A,B)$ to denote the set of maps $f:A \rightarrow B$ we deduce from Lemma~\ref{lem:span} and Lemma~\ref{lem:annihilator} the following corollary.

\begin{corollary} \label{cor:injectivity}
For $k\in\{0,\dots,\lfloor n/2 \rfloor\}$, the linear map $\Lambda^{n-2k}\R^n \rightarrow \calF(\Lambda^2\R^n, \Lambda^n \R^n)$ given by $\eta \mapsto (\omega \mapsto \omega^k \wedge \eta)$ is injective. For $k\in\{0,\dots,\lfloor (n-1)/2 \rfloor\}$, the linear map $\Lambda^{n-2k-1}\R^n  \rightarrow \calF(\Lambda^1\R^n \times \Lambda^2\R^n, \Lambda^n \R^n)$ given by $\eta \mapsto ((\theta,\omega) \mapsto\theta \wedge \omega^k \wedge \eta)$ is injective.
\end{corollary}

\begin{proof}
By linearity we only need to verify that these maps have trivial kernel. Let $k\in\{0,\dots,\lfloor n/2 \rfloor\}$ and $\eta\in \Lambda^{n-2k}\R^n$ be such that $\omega^k \wedge \eta = 0$ for all $\omega \in \Lambda^2\R^n$. On the one hand, by~\eqref{e:span-2k} we have $\eta \in \anhk[n-2k] \Lambda^{2k} \R^n$. On the other hand, by Lemma~\ref{lem:annihilator} we have $\anhk[n-2k] \Lambda^{2k} \R^n = \{0\}$. Therefore $\eta=0$. Similarly, for $k\in\{0,\dots,\lfloor (n-1)/2 \rfloor\}$ and $\eta \in \Lambda^{n-2k-1}\R^n$ such that $\theta \wedge \omega^k \wedge \eta = 0$ for all $(\theta,\omega) \in \Lambda^1\R^n \times \Lambda^2\R^n$, we have by~\eqref{e:span-2k+1} and Lemma~\ref{lem:annihilator} that  $\eta \in \anhk[n-2k-1] \Lambda^{2k+1} \R^n = \{0\}$.
\end{proof}

The next proposition played a key role at the end of the proof of Proposition~\ref{prop:step2-free-case}.

\begin{proposition} \label{prop:eta-map}
For $n\geq 1$ the following holds. Let $k \in \{1,\dots,n\}$ and $\overline{\eta}:\Lambda^1\R^n \rightarrow \Lambda^{k}\R^n$ be linear. Assume that $\theta \wedge \overline{\eta}(\theta) = 0$ for all $\theta \in \Lambda^1\R^n$. Then there is $\eta \in \Lambda^{k-1}\R^n$ such that $\overline{\eta}(\theta) = \theta\wedge \eta$ for all $\theta \in \Lambda^1\R^n$.
\end{proposition}

\begin{proof} When $k=n$ every map $\overline{\eta}:\Lambda^1\R^n \rightarrow \Lambda^{n}\R^n$ satisfies the assumption $\theta \wedge \overline{\eta}(\theta) = 0$ for all $\theta \in \Lambda^1\R^n$. If $\overline{\eta}:\Lambda^1\R^n \rightarrow \Lambda^{n}\R^n$ is in addition assumed to be linear and $\eta_j \in \R$ is such that $\overline{\eta}(e^j) = \eta_j \, e^{(1,\dots, n)}$, then $\overline{\eta}(\theta) = \theta\wedge \eta$ for all $\theta \in \Lambda^1\R^n$ where $\eta := \sum_{j=1}^n \sigma_j \, \eta_j  \, e^{(1, \dots, n)\setminus (j)}$ and $\sigma_j \in \{-1,1\}$ is such that $e^{(1,\dots ,n)} = \sigma_j \, e^j \wedge e^{(1,\dots, n)\setminus (j)}$.

Let us now argue by induction on $n$. First, if $n=1$, the conclusion follows from the previous remark. Next, let $n\geq 2$. By the previous remark we only need to consider the case where $k \in \{1,\dots,n-1\}$. By linearity of $\overline{\eta}$, we only need to prove that there is $\eta\in\Lambda^{k-1}\R^n$ such that $\ol\eta(e^j)=e^j\wedge\eta$ for $j \in \{1,\dots,n\}$. Set $V_j:=\{x\in \R^n: e^j(x)=0\}$ for $j \in \{1,\dots,n\}$.

For $j\in \{1,\dots,n-1\}$, since $e^j\wedge\ol\eta(e^j) =0$, we can write $\ol\eta(e^j)=e^j\wedge \eta_j$ for some $\eta_j\in\Lambda^{k-1}V_j$. Next, write $\eta_j=\tau_j\wedge e^n+ \zeta_j $ for some $\tau_j\in\Lambda^{k-2}V_n$ and $\zeta_j\in\Lambda^{k-1}V_n$ (when $k=1$, $\tau_j =0$). Define $\ol\tau:\Lambda^1V_n\to\Lambda^{k-1}V_n$ and $\ol\zeta: \Lambda^1 V_n \to\Lambda^{k}V_n$ to be linear and such that $\ol\tau(e^j)=e^j\wedge \tau_j$ and $\ol\zeta(e^j)=e^j\wedge\zeta_j$ for all $j\in \{1,\dots,n-1\}$. For  $\theta\in\Lambda^1 V_n$, we have $\ol\eta(\theta) = \ol\tau(\theta) \wedge e^n+\ol\zeta(\theta)$. Thus $\theta\wedge\ol\eta(\theta)= \theta\wedge\ol\tau(\theta)\wedge e^n+\theta\wedge\ol\zeta(\theta) =0$ for all $\theta\in\Lambda^1 V_n$ and this implies in turn that $\theta\wedge\ol\tau(\theta)=0$ and $\theta\wedge\ol\zeta(\theta)=0$ for all $\theta\in\Lambda^1V_n$. By induction, there are $\tau\in \Lambda^{k-2}V_n$ (again $\tau=0$ when $k=1$) and $\zeta\in \Lambda^{k-1}V_n$ such that $\ol\tau(\theta) = \theta \wedge \tau$ and $\ol\zeta(\theta) = \theta \wedge \zeta$ for all $\theta \in \Lambda^1V_n$. It follows that $\ol\eta(\theta)= \theta \wedge (\tau \wedge e^n + \zeta)$ for all $\theta \in \Lambda^1 V_n$. We set $\eta:= \tau \wedge e^n + \zeta \in \Lambda^{k-1} \R^n$. To conclude the proof of the proposition, it remains to verify that $\ol \eta (e^n) = e^n \wedge \eta$. Since $1\leq  k \leq n-1$, this is equivalent to showing that $e^j \wedge \ol\eta(e^n) = e^j \wedge e^n \wedge \eta$ for all $j \in \{1,\dots,n\}$, see Lemma~\ref{lem:annihilator}. By assumption we have $\theta\wedge\ol\eta(\theta)=0$ for all $\theta \in \Lambda^1\R^n$ and therefore $\theta\wedge\ol\eta(\theta') + \theta'\wedge\ol\eta(\theta)=0$ for all $\theta, \theta' \in \Lambda^1\R^n$. It follows that $e^j \wedge \ol \eta (e^n) = - e^n \wedge \ol \eta (e^j) = -e^n \wedge e^j  \wedge \eta =  e^j \wedge e^n \wedge \eta$  for $j\in \{1,\dots,n-1\}$. Since $e^n \wedge \ol \eta (e^n)=0 =  e^n \wedge e^n \wedge \eta$, we finally get $\ol \eta (e^n) = e^n \wedge \eta$, and this concludes the proof of the proposition.
\end{proof}

Recall that a graded ideal $\fraki$ of $\free_n$ can be seen as a linear subspace of $\Lambda^*\R^n$ of the form $\fraki = \fraki_1\oplus \fraki_2$ where $\fraki_1, \fraki_2$ are linear subspaces of respectively $\Lambda^1 \R^n, \Lambda^2\R^n$ such that $\theta \wedge \theta' \in \fraki_2$ for all $\theta \in \Lambda^1 \R^n$, $\theta'\in \fraki_1$. The structure of annihilators of such subsets of $\Lambda^*\R^n$, and in particular Corollary~\ref{cor:anh-homogeneous-ideal} and Lemma~\ref{lem:anh-graded-ideal}, played a major role in Section \ref{sect:h-affine-maps-arbitrary-case}. Before proving Corollary~\ref{cor:anh-homogeneous-ideal} and Lemma~\ref{lem:anh-graded-ideal}, we first state two elementary lemmas. The easy proof of Lemma~\ref{lem:anh-subsapces-in-direct-sum} is left to the reader.

\begin{lemma} \label{lem:anh-subsapces-in-direct-sum}
Let $V, W$ be linear subspaces of $\Lambda^*\R^n$ that are in direct sum. Then $\anh (V \oplus W) = \anh V \cap  \anh W$ and $\anhk[k] (V \oplus W) = \anhk[k] V \cap  \anhk[k] W$ for all $k\geq 0$.
\end{lemma}

\begin{lemma} \label{lem:anh-subsets-of-Lambdak}
Let $n\geq 1$, $j\geq 0$, $A\subset \Lambda^j\R^n$. Then $\anh A = \bigoplus_{k\geq 0} \anhk[k] A$.
\end{lemma}

\begin{proof}
Clearly, $\bigoplus_{k\geq 0} \anhk[k] A \subset \anh A$. Conversely, let $\eta = \sum_{k=0}^n \eta_k \in \anh A$ where $\eta_k \in \Lambda^k \R^n$. For $a\in A$, we have $\sum_{k=0}^n \eta_k \wedge a=0$ with $\eta_k \wedge a \in \Lambda^{k+j} \R^n$ therefore $\eta_k \wedge a =0$ for all $k\in \{0,\dots,n\}$. Since this holds true for all $a\in A$, we get that $\eta_k \in \anhk[k] A$ for all $k\in \{0,\dots,n\}$, and therefore $\eta \in \bigoplus_{k\geq 0} \anhk[k] A$.
\end{proof}

\begin{corollary} \label{cor:anh-homogeneous-ideal}
Let $\fraki_1, \fraki_2$ be linear subspaces of  respectively $\Lambda^1 \R^n, \Lambda^2\R^n$. Then $\anh (\fraki_1 \oplus \fraki_2) = \bigoplus_{k\geq 0} \anhk[k] (\fraki_1 \oplus \fraki_2)$.
\end{corollary}

\begin{proof} By Lemma~\ref{lem:anh-subsapces-in-direct-sum} and Lemma~\ref{lem:anh-subsets-of-Lambdak}, we have
\begin{equation*}
\begin{split}
\anh (\fraki_1 \oplus \fraki_2) &= \anh \fraki_1 \cap \anh \fraki_2 = (\oplus_{k\geq 0} \, \anhk[k] \fraki_1 ) \cap (\oplus_{k\geq 0} \, \anhk[k] \fraki_2)\\
& = \oplus_{k\geq 0} \, \anhk[k] \fraki_1 \cap \anhk[k] \fraki_2 = \oplus_{k\geq 0} \, \anhk[k] (\fraki_1 \oplus \fraki_2)~.
\end{split}
\end{equation*}
\end{proof}

\begin{lemma} \label{lem:anh-graded-ideal}
Let $n\geq 2$, $\fraki_1, \fraki_2$ be linear subspaces of respectively $\Lambda^1 \R^n, \Lambda^2\R^n$ such that $\theta \wedge \theta' \in \fraki_2$ for all $\theta \in \Lambda^1 \R^n$, $\theta'\in \fraki_1$. Then
\begin{align}
&\anhk[n] (\fraki_1 \oplus \fraki_2) = \Lambda^n \R^n, \label{e:anh-0}\\
&\anhk[n-1] (\fraki_1 \oplus \fraki_2) = \anhk[n-1] \fraki_1, \label{e:anh-1}\\
&\anhk[n-i] (\fraki_1 \oplus \fraki_2) = \anhk[n-i] \fraki_2 \, \text{ for } i \in \{2,\dots,n\}. \label{e:anh-i}
\end{align}
\end{lemma}

\begin{proof}
Clearly, $\anhk[n] A = \Lambda^n \R^n$ for all $A\subset \Lambda^{\geq 1}\R^n$ and~\eqref{e:anh-0} follows. By Lemma~\ref{lem:anh-subsapces-in-direct-sum}, we have $\anhk[n-1] (\fraki_1 \oplus \fraki_2) = \anhk[n-1] \fraki_1 \cap \anhk[n-1] \fraki_2$. For $A\subset \Lambda^{\geq 2}\R^n$, we have $\anhk[n-1] A = \Lambda^{n-1} \R^n$ therefore $\anhk[n-1] \fraki_2 = \Lambda^{n-1} \R^n$ and~\eqref{e:anh-1} follows. For $i \in \{2,\dots,n\}$, clearly $\anhk[n-i] (\fraki_1 \oplus \fraki_2)  \subset \anhk[n-i] \fraki_2$. Conversely, let $\eta \in \anhk[n-i] \fraki_2$ and $\theta \in \fraki_1$. By assumption, $\theta \wedge \theta' \in \fraki_2$ for all $\theta'\in \Lambda^1\R^n$. Therefore $\eta \wedge \theta \wedge \theta' = 0$ for all $\theta'\in \Lambda^1\R^n$, i.e., $\eta \wedge \theta \in \anhk[n-i+1] \Lambda^1\R^n = \Lambda^{n-i+1} \R^n \cap \Lambda^{\geq n} \R^n$, where the last equality follows from Lemma~\ref{lem:annihilator}. Since $i\in \{2,\dots,n\}$, we have $n-i+1 \leq n-1$, therefore $\Lambda^{n-i+1} \R^n \cap \Lambda^{\geq n} \R^n =\{0\}$ and hence $\eta \wedge \theta =0$. Since this holds true for all $\theta \in \fraki_1$, we get that $\eta \in \anhk[n-i] \fraki_1  \cap \anhk[n-i] \fraki_2  = \anhk[n-i] (\fraki_1 \oplus \fraki_2)$, where the last equality comes from Lemma~\ref{lem:anh-subsapces-in-direct-sum} and concludes the proof of~\eqref{e:anh-i}.
\end{proof}

We end this section with an observation that has been useful for our purposes in Remark~\ref{rmk:affhj>=i} and Section~\ref{sect:h-affine-maps-are-affine}.

\begin{lemma} \label{lem:anhk=0-anhj=0-j<=k}
Let $n\geq 1$, $k\in \{0,\dots,n\}$, $A \subset \Lambda^* \R^n$ be such that $\anhk[k] A = \{0\}$. Then $\anhk[j] A =\{0\}$ for all $j\in \{0,\dots,k\}$.
\end{lemma}

\begin{proof}
Let $j\in \{0,\dots,k\}$, $\eta \in \anhk[j] A$. For all $\zeta \in \Lambda^{k-j} \R^n$, $a\in A$, we have $\zeta \wedge \eta \wedge a = 0$, i.e., $\zeta \wedge \eta \in \anhk[k] A$. Since $\anhk[k] A = \{0\}$, it follows that $\eta \in \anhk[j] \Lambda^{k-j} \R^n = \Lambda^j \R^n \cap \Lambda^{\geq n-k+j+1} \R^n$, where the last equality follows from Lemma~\ref{lem:annihilator}. Since $n-k\geq 0$, we have $n-k+j+1 \geq j+1$, therefore $\Lambda^j \R^n \cap \Lambda^{\geq n-k+j+1} \R^n = \{0\}$ and hence $\anhk[j] A =\{0\}$.
\end{proof}

\bibliography{bibliography} 
\bibliographystyle{amsplain}

\end{document}